\documentclass{amsart}
\usepackage{graphicx} 
\usepackage{amsmath}  
\usepackage{amsthm}
\usepackage{amsfonts}
\usepackage{amssymb}
\usepackage{comment}
\usepackage{enumerate}
\usepackage{xurl}
\usepackage[colorlinks=true,linkcolor=blue,citecolor=blue,urlcolor=blue]{hyperref}   
\usepackage{longtable}
\usepackage{placeins}  

\def\I{{\bf \hat i}}
\def\J{{\bf \hat j}}

\def\Erdos{Erd\H{o}s}

\usepackage{makecell}

\theoremstyle{plain}
\newtheorem{theorem}{Theorem}
\newtheorem{lemma}[theorem]{Lemma}

\newtheorem{corollary}[theorem]{Corollary}

\theoremstyle{definition}
\newtheorem{definition}{Definition}[section]

\theoremstyle{remark}
\newtheorem*{remark}{Remark}

\def\Q{{\mathbb Q}}

\def\onehalf{\tfrac 1 2}

\title{Tiling a triangle into a prime number of congruent triangles}
\author[Beeson]{Michael Beeson}
\address{San Jos\'e State University (emeritus), and UCSC (research associate)}
\author[Roșu]{Grigore Roșu}
\address{University of Illinois at Urbana-Champaign}
\begin{document}

\begin{abstract}
We prove that if a triangle is cut into $N$ congruent triangles and $N$ is prime, then
$N=2$, $N=3$, or $N\equiv 1 \pmod 4$, and every such $N$ occurs.
In particular, if $N > 3$ then $N$ is not
a prime congruent to $3$ modulo $4$.  All the proofs in this paper
have been formalized in Lean.  

The argument divides into cases according to possible shapes of the tiled triangle.
Let $\alpha,\beta,\gamma$ be the angles of the tile.  Prior to this posting,
there remained several open shapes.  The two most difficult shapes were
$(2\alpha,\beta,\alpha+\beta)$, the {\em triquadratic} shape, and
$(\beta,\beta,3\alpha)$, the {\em isosceles-$\beta$} shape, in which
$3\alpha+2\beta = \pi$ and $\alpha$ is not a rational multiple of $\pi$.
A previous version of this paper relied on incorrect
proofs of those cases,  which have been retracted,  and are replaced
here by correct proofs. For each of those two shapes the number of tiles has the form $N = j^2 N_0$, where
$N_0 > 1$ is determined by the shape of the tile and $j$ is a positive integer;  so $N$
is prime, or even squarefree, only if $j = 1$.  We prove by geometric reasoning that neither shape admits a
tiling with $j = 1$, which shows that for both shapes $N$ is never prime, and indeed
never squarefree.

The shape $(\beta,\beta,3\alpha+\beta)$ (where $3\alpha + 3\beta = \pi$) has been settled by Bonfioli.
In this paper, we settle the remaining cases, by number-theoretic arguments
based on the area equation.  Those proofs rule out $N$ prime, but not $N$
squarefree, and we give examples to show that tilings with squarefree $N$
do sometimes exist.  
\end{abstract}

\maketitle

\section{Introduction}

A {\em tiling} of a triangle is a finite set of triangles congruent to the tile,
with pairwise disjoint interiors, whose union is the triangle.
That definition allows mirror images of tiles to be used, and we note that
 tilings are not assumed to be edge to edge.

A triangle $T$ has {\em commensurable angles} if each of its angles is a rational multiple of $\pi$,
and 
{\em incommensurable angles} otherwise.  
In our work on \Erdos\ problem~633 \cite{erdos-633, BLZ2026}, we had to show that triangles $T$ with
incommensurable angles have non-square tilings.  We gave a formula for each
possible shape of $T$, such that the
number of tiles $N$ and shape of the tile $(\alpha,\beta,\gamma)$ are
determined by some rational parameters.  Any tiling by that tile must
satisfy the equation for some value(s) of the parameter(s), and at least
some (but not necessarily every) value of the parameters gives a tiling.
Here we use similar formulas for $N$, but instead of showing that $N$ is not square, we
have to show that $N$ is not prime, except for $N=3$  and the primes that are
sums of two squares.

There is a second \Erdos\ problem about triangle tiling, problem~634 \cite{erdos-634},
which asks for the characterization of all $N$ such that no triangle can be tiled by $N$
congruent triangles.   The ``prime 634'' problem asks for the characterization
of the prime numbers belonging to that set.  In this paper we solve that problem:
for prime $N$, there is no $N$-tiling if and only if $N$ is
greater than three and congruent to 3 mod 4.

As a matter of notation,  we will always use $(a,b,c)$ for the side lengths
of the tile, and $\alpha,\beta,\gamma$ for the angles opposite those sides.

We divide our work into several  cases.  First we separate the cases when $T$ is
isosceles or equilateral.  After that we divide into cases
 according as $T$ has commensurable angles or not.
 The first two cases have been studied in prior works
 \cite{beeson-isosceles, beeson-equilateral}, which will be
 discussed below.

 \begin{lemma} \label{lemma:RandT}
 Let triangle $T$ be tiled by triangle $R$.  If $R$ has commensurable angles, so does $T$.
 Conversely,  if $T$ is not equilateral and has commensurable angles, so does $R$.
 \end{lemma}
\begin{proof} The first claim is immediate, since the angles of $T$ are sums of angles of $R$.
For the second claim, see the first paragraph of the proof of Theorem~5.3 in \cite{laczkovich1995}.
\end{proof}

The possible tilings with commensurable angles have been understood since the dawn
of the subject, and as we review below, in that case $N$ can be prime only when
$N=2$ and $T$ is isosceles, or $N=3$ and $T$ is equilateral or a 30-60-90 right triangle.  The fundamental
observation of Laczkovich%
\footnote{
See Theorems~4.1, 5.1,
and 5.3 of \cite{laczkovich1995}.
}
 is that in the incommensurable case,
if $T$ is not equilateral and not similar to the tile, then
we must have either
\begin{itemize}
\item $\gamma = \frac \pi 2$ \hskip3.8cm and $T$ is isosceles with angles $(\beta,\beta,2\alpha)$, or
\item $3\alpha+2\beta = \pi$   \hskip2.8cm(Group 1), or
\item $3\alpha + 3 \beta = \pi$,  i.e., $\gamma = 2\pi/3$  \hskip0.5cm(Group 2), or
\item $3\alpha + \beta = \pi$ \hskip3cm and $T$ is isosceles with base angles $\alpha$
\end{itemize}
The last case has been studied in \cite{beeson-isosceles}, and shown not to have any prime tilings;
hence it will not be a concern in this paper.
In both Group~1 and Group~2,  $\gamma > \pi/2$:  in Group~2 because $\gamma = 2\pi/3$, and in Group~1
because $\gamma = \pi - \delta$ with $\delta = \alpha+\beta < \pi/2$.  In particular the
tile is never a right triangle, so Theorem~\ref{theorem:rationality} applies.
In each group,  there is a small finite number of possible shapes of $T$, since the
angles of $T$ have to be nonnegative integral combinations of $(\alpha,\beta,\gamma)$.
The argument is bulldozer-style at the top level:  we consider each possible shape
of tile and prove that $N$ cannot be prime.

A triangle has {\em commensurable sides} if the ratios of its sides are rational numbers.
If the tile has commensurable sides, it can be scaled so that the sides are integers.
There is a close connection between ``incommensurable angles'' and
``commensurable sides'':
\begin{theorem}\cite[Theorem 1.2]{beeson-zhang2026}
    \label{theorem:rationality}
 Let triangle $T$ be tiled by a tile $R$ such that
    \begin{itemize}
        \item $R$ is not similar to $T$;
        \item $R$ is not a right triangle;
        \item $R$ has incommensurable angles.
    \end{itemize}
    Then $R$ must have commensurable sides.
\end{theorem}

Once we have a tiling of some triangle $T$, and we know the sides of the tile $R$
are commensurable, then the sides of $T$ are commensurable too, since they are linear
integral combinations of the sides; after a suitable scaling the sides will be
integers with no common factor.  That is how number theory enters this geometrical subject.
We assume $T$ is not similar to $R$.

Our original plan was to prove that
\begin{itemize}
\item $(a,b,c)$ are commensurable, so we can assume they are integers.
\item The sides $(X,Y,Z)$ of $T$ are of the form $\lambda(X_0,Y_0,Z_0)$, where
$(X_0,Y_0,Z_0)$ have no common factor.
\item $\lambda$ is an integer.
\item Then equating the area of $T$ to that of $N$ tiles, we have $N = \lambda^2 f(a,b)$ for
some function $f$, different for each shape.
\item $f$ is integer-valued.
\item Finally, the value of $f(a,b)$ cannot be a prime.  Then if $N$ is prime, $\lambda = 1$ and $N=f(a,b)$, contradiction.
\end{itemize}

Comparing this plan to the solution of \Erdos\ problem 633 in \cite{BLZ2026}:  the
problem there was to show that $N$ is not a square,  so instead of showing $f(a,b)$
is not a prime, we had to show $Z^2 = f(a,b)$ is not solvable.  Therefore the theory
of elliptic curves came into the proof.  In this paper, the number theory is
simpler, and one hard part comes from isosceles $T$.  In the \Erdos\ problem,
isosceles $T$ is disposed of in one line, since every isosceles triangle can be
cut in two and two is not a square.

This plan worked for most of the tiling shapes, but for the isosceles-$\beta$ shape $(\beta,\beta,3\alpha)$, and
the closely-related triquadratic shape $(2\alpha,\beta,\alpha+\beta)$, it did not work, since the $f(a,b)$ in
the above equation {\em can} take prime values for those shapes.  The proof we present here is
a combination of geometry and arithmetic, using geometry to prove that prime tilings must witness
long relations such as $hc = pa + qb$, and number theory to rule such relations out.

\section{Prior results}

The first results about prime tilings seem to be the direct proofs in
\cite{beeson-noseven} that there is no $7$-tiling and no $11$-tiling; the method
used there does not extend to $19$.

The question of prime tilings of the five Group~1 shapes has been studied in
\cite{beeson-triangletiling3}, where prime tilings are ruled out for three of
the five shapes, leaving the shapes $(2\alpha, \beta,\alpha+\beta)$ and $(\beta,\beta,3\alpha)$
open.  For short: the {\em triquadratic} and the {\em isosceles-$\beta$} shapes.%
\footnote{Earlier versions of \cite{beeson-triangletiling3} claimed to prove
non-primality for all five shapes, but two of those proofs were incorrect.
The proofs of Theorems 16, 22 and 26, cited in Table~\ref{table:Group1},
are unaffected.  (Those theorem numbers are from the corrected version 4.)
Several AI-generated proofs prompted by different people
were produced, but they relied on the incorrect proofs.
}
Table~\ref{table:Group1} lists the three shapes for which non-primality was already known, together
with exact citations for those results.

\begin{table}[ht]
\caption{Group 1, $3\alpha + 2 \beta = \pi$}
\begin{center}
\begin{tabular}{c l}
Angles of $T$ &  citation for $N$ not prime\\
\hline
$(\alpha,2\alpha,2\beta)$  & \cite{beeson-triangletiling3}, Theorem~16 \\
$(\alpha+\beta,\alpha+\beta,\alpha)$  & \cite{beeson-triangletiling3}, Theorem~22\\
$(\alpha,\alpha,\alpha+2\beta)$  & \cite{beeson-triangletiling3}, Theorem~26
\end{tabular}
\end{center}
\label{table:Group1}
\end{table}%

A proof of non-primality for the shape $(\beta,\beta,3\alpha+\beta)$
was posted 27 June, 2026 by Bonfioli \cite[Theorem~42]{bonfioli2026}.
   The other Group~2 shapes are solved here.

\begin{theorem} \label{theorem:equilateral}  Let $T$ be an equilateral triangle $N$-tiled, for $N > 3$.
Then the number of tiles $N$ is not a prime.
\end{theorem}

\begin{proof} This is proved in \cite{beeson-equilateral}.  Laczkovich already
showed that $\gamma$ must be $\pi/3$ or $2\pi/3$ \cite{laczkovich1995, laczkovich2012}, see especially Theorem~3.3 of \cite{laczkovich2012}.
Specifically, Laczkovich showed that one of the following holds:

\medskip

(i) $(\pi/3,\pi/3,\pi/3)$  (equilateral)

(ii) $(\pi/6,\pi/6,2\pi/3)$
\smallskip

(iii) $(\pi/6, \pi/2, \pi/3)$
\smallskip

(iv) $(\alpha, \beta, \pi/3)$ with $\alpha$ not a rational multiple of $\pi$
\smallskip

(v)  $(\alpha,\beta,2\pi/3)$ with $\alpha$ not a rational multiple of $\pi$
\smallskip

The angles are commensurable in the first three cases, and incommensurable
in the last two.  The proofs for the first three cases are short enough to
give here.
The possible $N$ in the first three cases are known.  In case (i),
$N$ must be a square, and any square corresponds to a tiling.  (That
is true for any triangle $T$ and a tile similar to $T$.)
Each side of $T$ is a sum of whole tile sides. In case~(ii),
the tile has sides $x$, $x$, $x \sqrt 3$ and area $(\sqrt 3/4)x^2$, so
a side of $T$ has length $(p + q\sqrt 3)x$ with integers $p,q \ge 0$, the
area of $T$ is $(\sqrt 3/4)(p+q\sqrt 3)^2x^2$, and
$N = (p+q\sqrt 3)^2 = p^2 + 3q^2 + 2pq \sqrt 3$; since $N$ is rational, $pq= 0$,
so $N = p^2$ or $N = 3q^2$, prime only for $N=3$.
In case~(iii) the tile has sides $x$, $x\sqrt 3$, $2x$ and area $(\sqrt 3/2)x^2$,
so a side of $T$ has length $(u+q\sqrt 3)x$ with integers $u,q \ge 0$, and $N = (u+q\sqrt 3)^2/2$; again $uq=0$,
so $N = u^2/2$ or $N = 3q^2/2$.  That is, $N = 2n^2$ or $N = 6n^2$, prime only for $N=2$.

Explicit citations for the non-primality of $N$ are
given in Table~\ref{table:equilateral}.
\begin{table}[ht]
\caption{Equilateral $T$}
\begin{center}
\begin{tabular}{r l}
Angles  &  citation for $N$ not prime\\
\hline
$\alpha = \beta = \pi/3$  & \cite{snover1991} or \cite{soifer} \\
$\alpha = \beta = \pi/6$  & proved just above \\
$(\alpha,\beta,\gamma) =(\pi/6, \pi/2, \pi/3)$   & \cite{laczkovich2012}, Theorem~3.1  \\
$\gamma =\pi/3, \alpha \not\in \Q$ & \cite{beeson-equilateral}, Theorem~3\\
$\gamma = 2\pi/3, \alpha \not\in \Q$  & \cite{beeson-equilateral}, Theorem~6
\end{tabular}
\end{center}
\label{table:equilateral}
\end{table}
\end{proof}

A {\em reptiling} is a tiling in which the tile is similar to the tiled triangle.
\begin{theorem}[\cite{laczkovich1995,laczkovich2012}]
   \label{theorem:l1995-5.3} Suppose $T$ has commensurable angles and is tiled by $R$.
   If $T$ is not an isosceles (or equilateral) triangle, or is an isosceles right triangle,
   then the tiling is a reptiling.
\end{theorem}

\begin{proof} Following \cite{laczkovich1995}, let $c(T)$ be the
number of distinct non-similar triangles $R$ into which $T$ can be tiled.  Then $c(T) = 1$
means that $T$
can only be tiled by $R$ similar to $T$ (that is, a reptiling).
Theorem~5.3 of \cite{laczkovich1995} says that, under the assumptions that $T$ has
commensurable angles and either is not isosceles, or is the isosceles right triangle,
then $c(T) = 1$, and if $T$ is isosceles but neither right nor equilateral, then $c(T) = 2$.
Hence, under the assumptions of our theorem, $c(T) = 1$.
\end{proof}

The main result related to reptiling is:
\begin{theorem}
    \label{theorem:reptile} Suppose $T$ can be reptiled by $N$ tiles. Then $N$ must be one (or more) of:
    \begin{itemize}
        \item $N = M^2$, which is possible for any triangle $T$ and any natural number $M$;
        \item $N = M^2 + K^2$, in which case, if $N$ is not a square,
         $T$ must be a right triangle with legs having ratio $M/K$;
        \item $N = 3M^2$, in which case $T$ must be the $(\pi/2, \pi/3, \pi/6)$ right triangle.
    \end{itemize}
\end{theorem}

\begin{proof} See the main theorem of \cite{snover1991} and its Corollaries~2 and~3; the ratio
of the legs is obtained in the proof of Lemma~4 there.
\end{proof}

Consequently:
\begin{corollary} \label{corollary:reptilenotprime}
Suppose $T$ can be reptiled by $N$ tiles and $N \not\in\{2, 3\}$ is prime.
Then $N \equiv 1 \pmod 4$, and
$T$ is a right triangle with legs in ratio $M/K$, where $N = M^2+K^2$.
\end{corollary}

\begin{proof}  Assume $N$ is a prime number.
Looking at the three cases of Theorem~\ref{theorem:reptile}, $N$ is a square, or three times a square, or
a sum of two squares.

In the square or three-times-a-square cases, $N$ is not prime, except when $N=3$.
  Therefore, $N$  is a sum
of two squares $M^2 + K^2$, and $T$ is a right triangle with legs in ratio $M/K$.  Therefore
$N \equiv 1 \pmod 4$, since $N \neq 2$ by hypothesis.
\end{proof}

\begin{remark} A 30-60-90 triangle can be 3-reptiled, and a right isosceles triangle is 2-reptiled by its altitude.
See Figure~\ref{figure:tiling13}, which also illustrates a biquadratic tiling with $N = 13 = 3^2+2^2$.
\end{remark}

\begin{figure}[ht]
\includegraphics[width=0.25\textwidth]{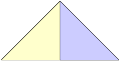}
\hskip0.5cm
\includegraphics[width=0.3\textwidth]{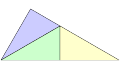}
\hskip0.5cm
\includegraphics[width=0.3\textwidth]{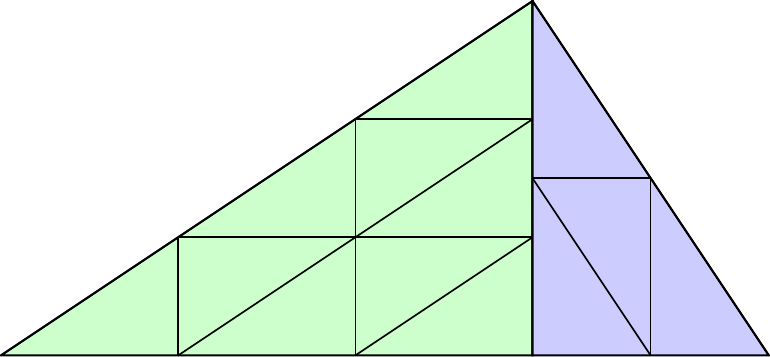}
\caption{Prime tilings}
\label{figure:tiling13}
\end{figure}

The following is Theorem~3.1 of \cite{beeson-isosceles}, where it is assembled from several theorems of \cite{laczkovich1995};
it is not stated in that form by Laczkovich.

\begin{theorem}[\cite{beeson-isosceles}, Theorem~3.1]
\label{theorem:isosceles}
Let $T$ be an isosceles, non-equilateral triangle that is  $N$-tiled by a tile with
angles $(\alpha,\beta,\gamma)$.  Then the possible shapes of
$T$ and the tile are given by Table~\ref{table:isosceles}.
 In the table,
the triples giving the angles of the tile are $(\alpha, \beta, \gamma)$
after a suitable permutation, i.e., they are unordered triples.
If the angles of the tile are commensurable, then one of the first two lines applies.
Otherwise $\alpha$ is not a rational multiple of $\pi$.
\end{theorem}

\begin{table}[ht]
\begin{center}
\setlength{\extrarowheight}{.4em}
\begin{tabular}{rl}
$T$  &  the tile     \\
\hline
$(\beta,\beta,\pi-2\beta)$  &  similar to $T$ \\
$(\beta,\beta,2\alpha)$ &  $\gamma = \pi/2$ \\
$(\alpha,\alpha,\pi-2\alpha)$  & $\gamma = 2\alpha$ \\
$(\alpha,\alpha,\pi-2\alpha)$ & $\gamma = 2\pi/3$ \\
$(\alpha,\alpha,\alpha+2\beta)$ &$3\alpha + 2\beta = \pi$ \\
$(\beta,\beta,3\alpha)$ &$3\alpha + 2\beta = \pi$ \\
$(\alpha+\beta,\alpha+\beta,\alpha)$ &$3\alpha + 2\beta = \pi$
\end{tabular}
\end{center}
\caption{Possible tilings of isosceles triangles, according to Laczkovich.}
\label{table:isosceles}
\end{table}

\begin{corollary} \label{cor:isoprime}
Let $T$ be isosceles but not equilateral, and $N$-tiled by a tile $R$
with angles $(\alpha,\beta,\gamma)$, where $N$ is prime and $N > 3$.
Then $\alpha$ is not a rational multiple of $\pi$, and $T$ and $R$ are
given by the sixth line of Table~\ref{table:isosceles}; that is,
$3\alpha + 2\beta = \pi$ and $T$ has angles $(\beta,\beta,3\alpha)$.
\end{corollary}

\begin{proof}
By Theorem~\ref{theorem:isosceles}, $T$ and $R$ are given by one of the
seven lines of Table~\ref{table:isosceles}.  We dispose of the other six.

Line~1.  The tiling is a reptiling, so by Theorem~\ref{theorem:reptile},
$N = M^2$, or $N = M^2 + K^2$, or $N = 3M^2$ with $T$ the
$(\pi/6,\pi/3,\pi/2)$ triangle.  The last of these is not isosceles, and a
square is not prime.  In the remaining case $N$ is not a square, since it is
prime, so $T$ is a right triangle whose legs have ratio $M/K$.  Being also
isosceles, $M = K$ and $N = 2M^2$, which is prime only when $N = 2$,
excluded by hypothesis.

Line~2.  By Corollary~7.9 of \cite{beeson-isosceles}, $N$ is not prime.

Line~3.  By Theorem~11.7 of \cite{beeson-isosceles}, $N$ is not squarefree;
in particular, $N$ is not prime.

Line~4.  Here $3\alpha+3\beta = \pi$, and since $\pi - 2\alpha = \alpha + 3\beta$,
the angles of $T$ are $(\alpha,\alpha,\alpha+3\beta)$, which is the shape
$(\beta,\beta,3\alpha+\beta)$ with $\alpha$ and $\beta$ interchanged.  As
$3\alpha+3\beta=\pi$ is symmetric in $\alpha$ and $\beta$, that case is solved in
\cite{bonfioli2026}.

Lines 5 and 7.  Theorems~26 and~22 of \cite{beeson-triangletiling3},
as shown in Table~\ref{table:Group1}.

Finally, since lines~1 and~2 have been excluded, the angles of $R$ are not
commensurable, so $\alpha$ is not a rational multiple of $\pi$ by
Theorem~\ref{theorem:isosceles}.
\end{proof}

\begin{remark}
The results cited in the proof of Corollary~\ref{cor:isoprime} for lines~2
and~3 of Table~\ref{table:isosceles} give more than non-primality.  In the
case $\gamma = \pi/2$, $N$ is not twice a prime congruent to $3 \pmod 4$
except for $N = 6$
(\cite{beeson-isosceles}, Corollary~7.9); in the case $\gamma = 2\alpha$,
$N$ is not squarefree (\cite{beeson-isosceles}, Theorem~11.7).
\end{remark}

By Lemma~\ref{lemma:RandT}, if $T$ has incommensurable angles then so does $R$.
This allows us to use the following corollary of another result from \cite{laczkovich1995}:

\begin{theorem}[Corollary of Theorem 4.1, \cite{laczkovich1995}]
\label{theorem:l1995-4.1}
Let a non-isosceles triangle $T$ be tiled by $R$ where $R$ has incommensurable angles $(\alpha, \beta, \gamma)$. Then one of the following must hold, after possibly interchanging the names $\alpha$ and $\beta$ (and $a$ and $b$).
\begin{enumerate}[{\rm (i)}]
    \item the tiling is a reptiling;
    \item (Group 1): $3\alpha + 2\beta = \pi$ and $T$ has angles $(\alpha, 2\alpha, 2\beta)$ or $(2\alpha, \beta, \alpha + \beta)$; or
    \item (Group 2): $3\alpha + 3\beta = \pi$ (equivalently, $\gamma = 2\pi/3$) and $T$ has angles
    $$(\alpha, 2\alpha, 3\beta), (\alpha, 2\beta, 2\alpha + \beta), (\alpha, \alpha+\beta, \alpha+2\beta), \text{or }  (2\alpha, 2\beta, \alpha+\beta).$$
\end{enumerate}
\end{theorem}

\begin{proof}
See \cite{laczkovich1995}, Theorem~4.1.
That theorem has more cases in the conclusion than
shown here, but does not have our hypothesis that
$T$ is not isosceles.
By adding in the requirement that $T$ is not isosceles, we reduce the number of cases of the original result.
\end{proof}

\section{Some lemmas about a tile with a $2\pi/3$ angle}
Here we present some lemmas that are about the tile and not about the shape of the
tiled triangle, together with  lemmas about the side lengths of the tiled triangle
that will be used in several of the cases below.

\begin{lemma} \label{lemma:c120}
Let triangle $R$ have angles $(\alpha,\beta,\gamma)$ and sides $(a,b,c)$ opposite
those angles, respectively.  Then $\gamma = 2\pi/3$ if and only if
$$ c^2 = a^2 + ab + b^2. $$
In that case we also have
$$ c^2 - a^2 = b(a+b), \qquad c^2-b^2 = a(a+b), $$
and $\alpha$ and $\beta$ are both less than $\pi/3$.
\end{lemma}

\begin{proof}
By the law of cosines,
$$ c^2 = a^2 + b^2 - 2ab \cos \gamma. $$
If $\gamma = 2\pi/3$ then $\cos \gamma = -\tfrac 1 2$, so $c^2 = a^2 + ab + b^2$.
Conversely, if $c^2 = a^2+ab+b^2$ then
$$ \cos \gamma = \frac{a^2+b^2-c^2}{2ab} = \frac{-ab}{2ab} = -\frac 1 2, $$
and since $0 < \gamma < \pi$, $\gamma = 2\pi/3$.

The two displayed identities follow by subtracting $a^2$, respectively $b^2$,
from $$c^2 = a^2+ab+b^2.$$  Finally $\alpha + \beta = \pi - \gamma = \pi/3$, and
both are positive, so each is less than $\pi/3$.
\end{proof}

\begin{lemma} \label{lemma:claude}
Let triangle $R$ have angles $(\alpha,\beta,2\pi/3)$ and sides $(a,b,c)$ opposite
those angles, respectively. Then
\begin{eqnarray*}
\cos \alpha &=& \frac{a+2b} {2c} \\
\sin \alpha &=& \frac {a \sqrt 3}{2c}\\
\sin 2 \alpha &=& \frac {a \sqrt 3 (a+2b)}{2c^2}\\
\sin 3 \alpha
&=&  \frac {a \sqrt 3}{2c} \cdot\frac {3b(a+b)}{c^2}
\end{eqnarray*}
\end{lemma}

\begin{proof}

By the law of sines, $\sin \alpha = (a/c) \sin (2\pi/3) = \frac {a \sqrt 3} {2c}$, as claimed.

To evaluate $\cos \alpha$:
\begin{eqnarray*}
\cos^2 \alpha &=& 1- \sin^2 \alpha = 1- \frac {3a^2}{4c^2} = \frac{4c^2-3a^2}{4c^2} = \frac{4(a^2+b^2+ab)-3a^2}{4c^2} \\
\cos^2 \alpha  &=& \frac{a^2 + 4b^2+4ab}{4c^2} = \frac{(a+2b)^2}{4c^2}
\end{eqnarray*}
By Lemma~\ref{lemma:c120}, we have $\alpha < \pi/3$, so $\cos(\alpha) > 0$. Hence we can take the square root:
$$ \cos \alpha = \frac {a+2b} {2c}, $$
as claimed.  Then we have
$$ \sin 2\alpha = 2 \sin \alpha \cos \alpha = 2\frac {a \sqrt 3} {2c} \cdot\frac{a+2b} {2c} = \frac {a \sqrt 3 (a+2b)}{2c^2}, $$
and finally
\begin{eqnarray*}
\sin 3 \alpha &=& \sin \alpha (3 - 4 \sin^2 \alpha) \\
&=&  \frac {a \sqrt 3}{2c} \left(3 - \frac {3a^2}{c^2}\right)  = \frac {a \sqrt 3}{2c} \left( \frac{3c^2 - 3a^2}{c^2}\right)\\
&=&  \frac {a \sqrt 3}{2c} \cdot \frac {3b(a+b)}{c^2} \qedhere
\end{eqnarray*}
\end{proof}

\begin{lemma}\label{lemma:beta}
Let triangle $R$ have angles $(\alpha,\beta,2\pi/3)$ and sides $(a,b,c)$ opposite
those angles, respectively. Then
\begin{eqnarray*}
\sin \beta &=& \frac{b \sqrt 3}{2c}\\
\cos \beta &=& \frac {2a+b}{2c}\\
\sin 2 \beta &=& \frac {b \sqrt 3 (2a+b)}{2c^2}\\
\sin (2 \alpha +\beta) &=& \frac{\sqrt 3 (a+b)}{2c} \\
\sin (\alpha + 2 \beta)&=&\frac{\sqrt 3 (a+b)}{2c}
\end{eqnarray*}
\end{lemma}

\begin{proof}
The first three formulas can be obtained by switching the labels $a$ and $b$, and $\alpha$ and $\beta$, in Lemma~\ref{lemma:claude},
which is legal since the requirement $3\alpha + 3\beta = \pi$ is invariant under that switch.  Then,
$$ \sin(2\alpha + \beta) = \sin(2\pi/3 -\alpha)= \frac{\sqrt 3} 2 \cos \alpha + \frac 1 2 \sin \alpha  = \sqrt 3 \frac {a+b}{2c}. $$
Finally, $\sin(\alpha+2\beta) = \sin(2\alpha+\beta)$, since the sum of those two angles is $\pi$.
\end{proof}

\begin{lemma} \label{lemma:pairwisecoprimality} Suppose $c^2 = a^2 + b^2 + ab$ and
$(a,b,c)$ have no common divisor.  Then $(a,b,c)$ are pairwise relatively prime.
\end{lemma}
\begin{proof}  Suppose $p$ is prime and $p \mid a$ and $p \mid b$.  Then $p \mid a^2 + b^2 + ab$, so
$p \mid c$, so $p$ divides all three of $(a,b,c)$, contradicting the hypothesis.
  Suppose instead that $p \mid a$ and $p \mid c$.  Then $p \mid c^2 -a^2 - ab = b^2$,
so $p \mid b$, again contradicting the hypothesis.  Similarly with $a$ and $b$ switched.
\end{proof}

\begin{lemma} \label{lemma:not3divc} Let $(a,b,c)$ have no common factor
and satisfy $c^2 = a^2 + ab + b^2$.  Then $3 \nmid c$.
\end{lemma}

\begin{proof}
By Lemma~\ref{lemma:pairwisecoprimality}, $\gcd(a,c) = 1$.
Suppose $3 | c$. Then $9 | c^2 = a^2 + ab+b^2$.   Then $a^2 + ab + b^2 \equiv 0 \pmod 3$.
We will show $ a \equiv b \pmod 3$.  If $a \equiv 0$, then $0 \equiv a^2 + ab + b^2 \equiv b^2$, so $b \equiv 0 \equiv a$.
If $a \equiv 1 \pmod 3$, then $0 \equiv a^2 + ab + b^2 = 1+ b + b^2 \pmod 3$.  Then $b \equiv 1 \equiv a \pmod 3$.  If $a\equiv 2 \pmod 3$ then
$0 \equiv 1 + 2b + b^2 \pmod 3$, which implies $b \equiv 2 \equiv a \pmod 3$.  Therefore $a \equiv b\pmod 3$, as claimed.   Therefore
there exists an integer $t$ such that $b = a + 3 t$.   Then
$$ c^2 = a^2 + ab + b^2 = 3 a^2 + 9 at + 9 t^2  \equiv 3a^2 \pmod 9.$$
Since $3 \mid c$ and $\gcd(a,c) = 1$, we have $3 \nmid a$, so $a^2 \equiv 1 \pmod 3$,
and hence $3a^2 \equiv 3 \pmod 9$.  Therefore $c^2 \equiv 3 \pmod 9$, so $9 \nmid c^2$.
But $3 \mid c$ implies $9 \mid c^2$, a contradiction.
\end{proof}

\begin{lemma} \label{lemma:BLZ30} Suppose $c^2 = a^2 + b^2 + ab$ and
$(a,b,c)$ have no common divisor.  Then
$$ \gcd(c^2,c(a+2b),3b(a+b)) = 1.$$
\end{lemma}

\begin{proof}  Let a prime $p$ divide $c^2$, $c(a+2b)$, and $3b(a+b)$.  Then $p \mid c$.
If $p=3$, this contradicts Lemma~\ref{lemma:not3divc}.  If $p \neq 3$, then
$p \mid b(a+b)$.  By Lemma~\ref{lemma:pairwisecoprimality}, $p \nmid b$, so $p \mid a+b$.
Then $c^2 = a^2+b(a+b)$ gives $p \mid a$, and with $p \mid a+b$ also $p \mid b$,
contradicting Lemma~\ref{lemma:pairwisecoprimality}.
\end{proof}

\begin{lemma} \label{lemma:claude4} Let  $(a,b,c)$ be positive integers
satisfying $c^2 = a^2 + ab + b^2$.  Then $a+b$ is composite.
\end{lemma}

\begin{proof}  Suppose $p := a+b$ is prime.  Then
$c^2 = p^2 - ab$, and  we have
\begin{eqnarray*}
0 < c < p &&   \mbox{\qquad since $c^2 = p^2-ab$ and $ab > 0$}\\
b \equiv -a \pmod p   &&\mbox{\qquad since $p:= a+b$}\\
c^2 \equiv p^2-ab \equiv a^2 \pmod p && \mbox{\qquad since $b \equiv -a \pmod p$} \\
p \mid (c-a)(c+a)  && \mbox{\qquad since $c^2 - a^2 = (c-a)(c+a)$}\\
|c-a| < p  &&  \mbox{\qquad since $0 < a < p$ and $0 < c < p$} \\
c\neq a && \mbox{\qquad since if $c=a$ then $b(a+b)=0$, which is false}\\
p \nmid (c-a)  && \mbox{\qquad by the preceding lines}\\
p \mid (c+a)     && \mbox{\qquad since $ p \mid (c-a)(c+a)$ but $p \nmid (c-a)$}\\
c+a = p = a+b   && \mbox{\qquad since $0 < c+a < 2p$}\\
c = b    && \\
a(a+b) = 0  && \mbox{\qquad since $c^2 = a^2 + ab + b^2$ and $c=b$}
\end{eqnarray*}
But that is impossible, since $a$ and $b$ are positive.
\end{proof}
\begin{lemma} \label{lemma:claude3}
Suppose triangle $T$ is tiled by a triangle $R$ with integer sides.
 Let the side lengths
$(X,Y,Z)$ of $T$ be proportional to coprime integers $(X_0, Y_0, Z_0)$.
Then $(X,Y,Z) = \lambda (X_0,Y_0,Z_0)$ for some integer $\lambda$.
\end{lemma}

\begin{proof} Since there is a tiling, each side length of $T$ is a
linear integral combination of $(a,b,c)$, which are integers by
hypothesis.  Hence $(X,Y,Z)$ are integers.  By hypothesis,
$$ (X,Y,Z) = \lambda (X_0,Y_0,Z_0)$$
for some $\lambda$.  Then $\lambda$ is rational, since $X$ and $X_0$
are integers.  Let $p,q$ be coprime positive integers with $\lambda = p/q$.
We have
$$ q (X,Y,Z) = p (X_0,Y_0,Z_0).$$
Since $p$ and $q$ are coprime, $q$ divides each of $(X_0$, $Y_0$, $Z_0)$.  But
by hypothesis, these three have no common factor.  Hence $q = 1$ and $\lambda = p$
is an integer.
\end{proof}

\begin{lemma}\label{lemma:rationality3}
Suppose triangle $T$ with incommensurable angles is tiled using triangle $R$ with angles $(\alpha, \beta, 2\pi/3)$.
Suppose $T$ is not similar to $R$.
Then the side lengths $(a,b,c)$ of $R$ are commensurable.
\end{lemma}

\begin{proof}
$R$ is not a right triangle, since $\gamma = 2\pi/3$ and $\alpha, \beta < \pi/3$.
$T$ is not equilateral, since its angles are incommensurable.
$R$ has incommensurable angles, by Lemma \ref{lemma:RandT}.
$R$ is not similar to $T$, by hypothesis.  Then by Theorem~\ref{theorem:rationality},
$(a,b,c)$ are commensurable.
\end{proof}

\section{The shape $(\alpha,2\alpha,3\beta)$}

\begin{theorem} \label{theorem:case2}  Let triangle $T$ have incommensurable angles
$(\alpha, 2\alpha, 3\beta)$, and suppose $T$ is $N$-tiled  by $(\alpha,\beta,2\pi/3)$.
Then $N$ is not a prime.  Moreover,
$$ N =  3\lambda^2 (a+2b)  (a+b),$$
where $(a,b,c)$ are integers with no common factor, proportional to the sides of the tile, and $\lambda$ is a positive integer.
\end{theorem}

\begin{proof} $T$ has no $2\pi/3$ angle, since if $3\beta = 2\pi/3$ then $\beta = 2\pi/9$,
violating the incommensurable angles hypothesis; and if $2\alpha = 2\pi/3$ then $\alpha = \pi/3$,
also violating the incommensurable angles hypothesis; and if $\alpha = 2\pi/3$ then $2\alpha > \pi$
could not be an angle of $T$.
Since $T$ has no $2\pi/3$ angle, it is not similar to the tile.
By Lemma~\ref{lemma:rationality3}, $(a,b,c)$ are commensurable.
Without loss of generality, we may assume they are integers with no common factor.
Then the side lengths $(X,Y,Z)$ are also integers, since they are integral linear
combinations of $(a,b,c)$.
Let the sides of $T$ next to the $3\beta$ angle be $X$ and $Y$.
By the law of sines,
$$ X: {\sin \alpha} =   Y :{\sin 2 \alpha} = Z :{\sin 3 \beta}.$$
Since $\pi = 3 \alpha + 3\beta$, we have $\sin 3\beta = \sin 3 \alpha$.   Therefore
$$ X: {\sin \alpha} =   Y :{\sin 2 \alpha} = Z : {\sin 3 \alpha}.$$
By Lemma~\ref{lemma:claude},
$$ X : \frac {a \sqrt 3}{2c} = Y:\frac {a \sqrt 3 (a+2b)}{2c^2} = Z: \frac {a \sqrt 3}{2c} \cdot \frac{3b(a+b)}{c^2}.$$
Dividing by $a \sqrt{3}/2$ and multiplying by $c^3$ we have
$$ X : c^2 = Y: c(a+2b) = Z:  {3b(a+b)}.$$
That is, for some real number $\lambda$, we have
\begin{equation}
 (X,Y,Z) = \lambda(c^2,c(a+2b),3b(a+b)). \label{eq:421}
 \end{equation}
Since $X$ and $c$ are positive integers, $\lambda$ is a positive rational.
By Lemma~\ref{lemma:BLZ30}, $\gcd(c^2,c(a+2b),3b(a+b)) = 1$.
By Lemma~\ref{lemma:claude3}, $\lambda$ is a positive integer.
(This is the key step of the proof.)

 The area equation expresses that the area of $T$ is $N$ times the area of the tile:
\begin{equation*}
XY \sin 3\beta = N ab \sin \gamma \ = \ N ab \frac {\sqrt 3} 2
\end{equation*}

Substitute these values for $X$ and $Y$ from (\ref{eq:421}) into the area equation.
We obtain
\begin{eqnarray*}
XY \sin 3\beta &=& \ N ab \frac {\sqrt 3} 2  \\
\lambda^2 c^3(a+2b)\sin 3 \beta &=& Nab \frac{\sqrt 3} 2
\end{eqnarray*}
Since $3\alpha+3\beta =\pi$, we have $\sin 3 \beta = \sin 3 \alpha$.  Therefore
\begin{eqnarray*}
\lambda^2 c^3(a+2b)\sin 3 \alpha &=& Nab \frac{\sqrt 3} 2
\end{eqnarray*}
Substituting for $\sin 3 \alpha$ from Lemma~\ref{lemma:claude}, we have
\begin{eqnarray*}
\lambda^2 c^3(a+2b) \cdot \frac {a \sqrt 3}{2c} \cdot \frac {3b(a+b)}{c^2} &=& Nab \frac{\sqrt 3} 2 \\
 3\lambda^2 (a+2b)  (a+b) &=& N
\end{eqnarray*}
Since $a \ge 1$ and $b \ge 1$ we have $a+b \ge 2$ and $a + 2b \ge 3$, so
$(a+2b)(a+b) \ge 6$.  Hence $3 \mid N$ and $N \ge 18$, so $N$ is not prime.
\end{proof}

\begin{remark}
The formula $N = 3\lambda^2(a+2b)(a+b)$ gives more than non-primality.  Since
$\gcd(a+b,a+2b) = \gcd(a,b) = 1$, the number $N$ is squarefree only if $\lambda = 1$ and
$a+b$ and $a+2b$ are both squarefree and not divisible by $3$.  Moreover $3 \mid ab$ is
then forced:  if $a \equiv b \pmod 3$ then $c^2 = a^2+ab+b^2 \equiv 3a^2 \equiv 0$, so
$3 \mid c$, contrary to Lemma~\ref{lemma:not3divc};  so if neither $a$ nor $b$ is
divisible by $3$ then $\{a,b\} \equiv \{1,2\} \pmod 3$ and $3 \mid a+b$, whence
$9 \mid N$.  Thus $N$ is not squarefree unless $3 \mid ab$.  That case does occur:  the
smallest instance is the tile $(24,11,31)$ with $\lambda = 1$, giving
$N = 3\cdot46\cdot35 = 4830$, which is squarefree.  See Figure~\ref{figure:4830}.
There are smaller tilings with $\lambda = 1$ in which $N$ is not squarefree;  see Figure~\ref{figure:312}.

\begin{figure}[ht]
\begin{center}
\includegraphics[width=\textwidth]{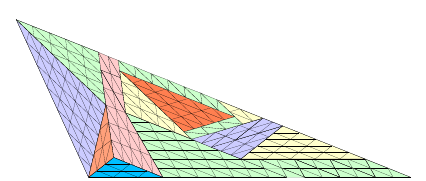}
\end{center}
\caption{$N = 312 = 2^3\cdot 3 \cdot 13$. Here $(3,5,7)$ tiles $(49,91,120)$.}
\label{figure:312}
\end{figure}

\end{remark}

\section{The shape $(\alpha,2\beta,2\alpha+\beta)$}

\begin{theorem} \label{theorem:case3}  Let triangle $T$ have incommensurable angles
$(\alpha, 2\beta, 2\alpha + \beta)$, and suppose $T$ is $N$-tiled by $(\alpha,\beta,2\pi/3)$.
Then $N$ is not a prime.  Moreover,
$$ N = \lambda^2 (2a+b) (a+b),$$
where $(a,b,c)$ are integers with no common factor, proportional to the sides of the tile, and $\lambda$ is a positive integer.
\end{theorem}

\begin{proof}  Since $\alpha, \beta < \pi/3$,  neither $\alpha$ nor $2\beta$ is equal to $2\pi/3$.
If $2\alpha + \beta = 2\pi/3$ then $\alpha = \pi/3$, since $\alpha + \beta = \pi/3$; but that
violates the incommensurable angles hypothesis.  Hence $T$ has no $2\pi/3$ angle.
Therefore $T$ is not similar to the tile.
By Lemma~\ref{lemma:rationality3}, $(a,b,c)$ are commensurable.
Without loss of generality, we may assume they are integers with no common factor.
Then the side lengths $(X,Y,Z)$ are also integers, since they are integral linear
combinations of $(a,b,c)$.
Let the sides of $T$ next to the $2\alpha + \beta$ angle be $X$ and $Y$.
We will show that $(X,Y,Z)$  are proportional to a triple with
no common factor, apply Lemma~\ref{lemma:claude3} to conclude
that the proportionality factor is an integer,
and then use the area equation to finish.   Here are the details:

By the law of sines,
$$ X: {\sin \alpha} =   Y :{\sin 2 \beta} = Z :{\sin( 2\alpha + \beta)}.$$
Using the formulas from Lemmas~\ref{lemma:claude} and~\ref{lemma:beta}, we have
$$ X: \frac {a \sqrt 3}{2c} = Y:\frac {b \sqrt 3 (2a+b)}{2c^2} = Z: \frac{\sqrt 3 (a+b)}{2c}.$$
Dividing the second entry of each ratio by $\sqrt 3/2$ and multiplying it by $c^2$,
we have
$$ X: ac = Y : b(2a+b) = Z: (a+b)c.$$
Then for some real $\lambda$ we have
\begin{equation}
 (X,Y,Z) = \lambda (ac, b(2a+b), (a+b)c). \label{eq:518}
 \end{equation}
We will prove $(ac,b(2a+b),(a+b)c)$ have no common factor.
We have $c^2 = a^2+ab+b^2$, so by Lemma~\ref{lemma:pairwisecoprimality},
$(a,b,c)$ are pairwise relatively prime.
Suppose the prime $p$ divides all three of $(ac,b(2a+b),(a+b)c)$.
  Then $p | ac$, so
$p | a$ or $p | c$.  Then $p \nmid b$, since $p$ divides $a$ or $c$,
each of which is coprime to $b$.  But $p | b(2a+b)$,
so $p | (2a+b)$.  If $p | a$ then $p | b$; contradicting $\gcd(a,b) = 1$.
Therefore $p \nmid a$.  Therefore $p | c$, since $p | a$ or $p | c$.
Modulo $p$ we have
\begin{eqnarray*}
b \equiv -2a \pmod p  && \mbox{\qquad since $p | (2a+b)$}\\
0 \equiv c^2 \equiv a^2 + b^2 + ab  \pmod p&& \mbox{\qquad since $p | c$}\\
0  \equiv 3a^2 \pmod p  && \mbox{\qquad since $b \equiv -2a$}\\
p = 3                && \mbox{\qquad since $p \nmid a$}
\end{eqnarray*}
But that contradicts Lemma~\ref{lemma:not3divc}.
Therefore $(ac,b(2a+b),(a+b)c)$ have no common factor.
Then by Lemma~\ref{lemma:claude3}, $\lambda$ is a positive integer in (\ref{eq:518}).

The area equation for this shape of triangle is
$$ XY \sin (2\alpha+\beta) \ = \ N ab \frac {\sqrt 3} 2.$$
Substituting the values for $X$ and $Y$ from (\ref{eq:518}) into the area equation,
we obtain
\begin{eqnarray*}
\lambda^2 abc(2a+b) \sin(2\alpha+\beta) &=& \ N ab \frac {\sqrt 3} 2  \\
\lambda^2 c(2a+b) \sin(2\alpha+\beta)  &=& N \frac {\sqrt 3} 2
\end{eqnarray*}
Putting in the value of $\sin(2\alpha+\beta)$ from Lemma~\ref{lemma:beta}, we have
\begin{eqnarray*}
\lambda^2 c(2a+b) \sqrt 3 \frac {a+b}{2c}  &=& N \frac {\sqrt 3} 2  \\
\lambda^2 (2a+b) (a+b)  &=& N
\end{eqnarray*}
Since $a \ge 1$ and $b \ge 1$, we have $a+b \ge 2$ and $2a+b \ge 3$.  Hence
$N = \lambda^2 (2a+b)(a+b)$ is a product of two integers each at least $2$,
so $N$ is not prime.
\end{proof}

\section{The shape $(\alpha, \alpha+\beta, \alpha+2\beta)$}
\begin{theorem} \label{theorem:case4}  Let triangle $T$ have incommensurable angles
$(\alpha, \alpha + \beta, \alpha + 2\beta)$, and suppose $T$ is $N$-tiled by $(\alpha,\beta,2\pi/3)$.
Then $N$ is not a prime.  Moreover,
$$ N = \frac {\lambda^2 (a+b)} b,$$
where $(a,b,c)$ are integers with no common factor, proportional to the sides of the tile, and $\lambda$ is a positive integer such that $b \mid \lambda^2$.
\end{theorem}

\begin{proof}
 Since $\alpha, \beta < \pi/3$,  neither $\alpha$ nor $\alpha + \beta$ is equal to $2\pi/3$.
If $\alpha + 2\beta = 2\pi/3$ then $\beta = \pi/3$, since $\alpha + \beta = \pi/3$; but that
violates the incommensurable angles hypothesis.  Hence $T$ has no $2\pi/3$ angle.
Therefore $T$ is not similar to the tile.
By Lemma~\ref{lemma:rationality3}, $(a,b,c)$ are commensurable.
Without loss of generality, we may assume they are integers with no common factor.
Then the side lengths $(X,Y,Z)$ are also integers, since they are integral linear
combinations of $(a,b,c)$.
Let the sides of $T$ next to the $\alpha + 2\beta$ angle be $X$ and $Y$.
Again we will show that $(X,Y,Z)$ are multiples of a triple with no common factor, and apply Lemma~\ref{lemma:claude3}
and then the area equation.   Here are the details:

By the law of sines,
$$ X: {\sin \alpha} =   Y :{\sin(\alpha + \beta)} = Z :{\sin( \alpha + 2\beta)}.$$
Using the formulas from Lemmas~\ref{lemma:claude} and~\ref{lemma:beta},  as well as $\alpha + \beta = \pi/3$, we have
$$ X: \frac {a \sqrt 3}{2c} = Y:\frac { \sqrt 3 }2 = Z: \frac{\sqrt 3 (a+b)}{2c} .$$
Multiplying the second entry of each ratio by $2c/\sqrt 3$, we have
$$ X: a = Y : c = Z: (a+b).$$
Then for some real $\lambda$ we have
\begin{equation}
 (X,Y,Z) = \lambda (a, c, a+b). \label{eq:575}
 \end{equation}
Now $(a, c, a+b)$ have no common factor, since if prime $p$ divides $a$, $c$, and $a+b$, then it also divides $b$,
 contradiction.
Then by Lemma~\ref{lemma:claude3}, $\lambda$ is a positive integer in (\ref{eq:575}).

The area equation for this shape of triangle is
$$ XY \sin (\alpha+2 \beta) \ = \ N ab \frac {\sqrt 3} 2.$$
Substituting the values for $X$ and $Y$ from (\ref{eq:575}) into the area equation,
we obtain
\begin{eqnarray*}
\lambda^2 ac \sin(\alpha+2 \beta) &=& \ N ab \frac {\sqrt 3} 2  \\
\lambda^2 c \sin(\alpha+2\beta)  &=& Nb \frac {\sqrt 3} 2
\end{eqnarray*}
Putting in the value of $\sin(\alpha+2\beta)$ from Lemma~\ref{lemma:beta}, we have
\begin{eqnarray*}
\lambda^2 c \sqrt 3 \frac {a+b}{2c}  &=& Nb\frac {\sqrt 3} 2  \\
N &=& \frac {\lambda^2 (a+b)} b
\end{eqnarray*}
We have $\gcd(a+b,b) = \gcd(a,b) = 1$, so $b$ divides $\lambda^2$; put $m = \lambda^2/b$,
a positive integer, so that $N = m(a+b)$.  By Lemma~\ref{lemma:claude4}, $a+b$ is
composite, say $a+b = de$ with $d > 1$ and $e > 1$.  Then $N = mde$, so $d$ is a
divisor of $N$ with $1 < d < N$.  Hence $N$ is not prime.
\end{proof}

\section{The shape $(2\alpha, 2\beta, \alpha+\beta)$}
\begin{theorem} \label{theorem:case5}  Let triangle $T$ have incommensurable angles
$(2\alpha, 2\beta, \alpha + \beta)$, and suppose $T$ is $N$-tiled by $(\alpha,\beta,2\pi/3)$.
Then $N$ is not a prime.  Moreover,
$$N = \lambda^2 (a+2b) (2a+b),$$
where $(a,b,c)$ are integers with no common factor, proportional to the sides of the tile,
and $\lambda$ is a positive integer.
\end{theorem}

\begin{remark} This is Conjecture~4 in \cite{zhang2026},  proved there
under the assumption $a \not\equiv b \pmod 3$.  That assumption is in fact
automatic: if $a \equiv b \pmod 3$ then $c^2 = a^2+ab+b^2 \equiv 3a^2 \equiv 0
\pmod 3$, so $3 \mid c$, contradicting Lemma~\ref{lemma:not3divc}.  Hence that
proof already covers all cases; we nevertheless give an independent proof.
\end{remark}

\begin{proof}[Proof of Theorem~\ref{theorem:case5}]
 Since $\alpha, \beta < \pi/3$,  neither $2\alpha$ nor $2\beta$ nor $\alpha+\beta$ is equal to $2\pi/3$.
Therefore $T$ is not similar to the tile.
By Lemma~\ref{lemma:rationality3}, $(a,b,c)$ are commensurable.
Without loss of generality, we may assume they are integers with no common factor.
Then the side lengths $(X,Y,Z)$ are also integers, since they are integral linear
combinations of $(a,b,c)$.
Let the sides of $T$ next to the $\alpha + \beta$ angle be $X$ and $Y$.
Again we will show that $(X,Y,Z)$ are multiples of a triple with no common factor, and apply Lemma~\ref{lemma:claude3}
and then the area equation.   Here are the details:

By the law of sines,
$$ X: {\sin 2\alpha} =   Y :{\sin 2\beta} = Z :{\sin (\alpha + \beta)}.$$
Using the formulas from Lemmas~\ref{lemma:claude} and~\ref{lemma:beta},  as well as $\alpha + \beta = \pi/3$, we have
$$ X: \frac {a \sqrt 3 (a+2b)}{2c^2} = Y:\frac {b \sqrt 3 (2a+b)}{2c^2} = Z: \frac{\sqrt 3 } 2 .$$
Multiplying the second entry of each ratio by $2c^2/\sqrt 3$, we have
$$ X: a(a+2b) = Y : b(2a+b) = Z: c^2.$$
Then for some real $\lambda$ we have
\begin{equation}
 (X,Y,Z) = \lambda (a(a+2b), b(2a+b), c^2). \label{eq:618}
 \end{equation}
 We will show that $(a(a+2b), b(2a+b), c^2)$ have no common factor.
 Suppose the prime $p$ divides all three. Then $p|c$, so $p \nmid a$ and $p \nmid b$,
 by Lemma~\ref{lemma:pairwisecoprimality}.
 Therefore $p \mid (2a + b)$.  Then
 \begin{eqnarray*}
 b \equiv -2a \pmod p   && \mbox{\qquad since $p \mid 2a+b$}\\
 0  \equiv 3a^2 \pmod p && \mbox{\qquad since $c^2 = a^2 + b^2 + ab$ and $c \equiv 0$ and $b \equiv -2a$}\\
 p = 3                   && \mbox{\qquad since $p \nmid a$} \\
 3 \nmid c               && \mbox{\qquad by  Lemma~\ref{lemma:not3divc}} \\
 p \nmid c               && \mbox{\qquad by the preceding two lines}
 \end{eqnarray*}
 But that contradicts the assumption $p \mid c^2$.  That completes
 the proof that  $(a(a+2b), b(2a+b), c^2)$ have no common factor.
Then by Lemma~\ref{lemma:claude3}, $\lambda$ is a positive integer in (\ref{eq:618}).

The area equation  is
$$ XY \sin (\alpha +\beta) \ = \ N ab \frac {\sqrt 3} 2.$$
Since $\alpha + \beta = \pi/3$, $\sin(\alpha+\beta) = \tfrac{\sqrt 3} 2$, so
the area equation becomes
$$ XY = N ab.$$
Substituting the values for $X$ and $Y$ from (\ref{eq:618}) into the area equation,
we obtain
\begin{eqnarray*}
\lambda^2 a(a+2b) b(2a+b)&=& \ N ab  \\
N &=& \lambda^2 (a+2b) (2a+b)
\end{eqnarray*}
Since $a \ge 1$ and $b \ge 1$, both $a+2b \ge 3$ and $2a+b \ge 3$.  Hence
$N = \lambda^2(a+2b)(2a+b)$ is a product of two integers each at least $2$,
so $N$ is not prime.
\end{proof}

\section{The isosceles shape $(\alpha,\alpha,\pi-2\alpha)$}

This is the remaining Group~2 shape, and the only one for which the method of the
preceding four sections fails.  It is settled in \cite{bonfioli2026}; we explain
here why a different argument is needed.

 For the isosceles shape $(\alpha,\alpha,\pi-2\alpha)$ the same
method yields a formula but no conclusion.  Writing $X$ for the equal sides and
$Y$ for the base, the law of sines gives $(X,X,Y) = \lambda(c,c,a+2b)$.
We claim that $\gcd(c,a+2b) = 1$. For
if a prime $p$ divided both, then
$a \equiv -2b \pmod p$ would give $0 \equiv c^2 = a^2+ab+b^2 \equiv 3b^2$,
so $p = 3$ or $p \mid b$;  the latter contradicts $\gcd(b,c) = 1$
(Lemma~\ref{lemma:pairwisecoprimality}) and the former contradicts
Lemma~\ref{lemma:not3divc}.  Hence $\gcd(c,a+2b) = 1$, as claimed.
Then $\lambda$ is a positive
integer by Lemma~\ref{lemma:claude3}; and the area equation
$$ 2X^2 \cos\alpha = N\,bc $$
then gives $N = \lambda^2(a+2b)/b$.  Since $\gcd(a+2b,b)=1$, we have
$b \mid \lambda^2$ and $N = m(a+2b)$ with $m = \lambda^2/b$ a positive integer.
But unlike $a+b$ (Lemma~\ref{lemma:claude4}), $a+2b$ need not be composite:
$N$ is prime exactly when $b$ is a square and $a+2b$ is prime.
The smallest parameters for which the area equation allows a prime are the tile $(5, 16, 19)$
with $\lambda=4$, equal sides $76$, base $148$,
 and $N = 37$.  This is why this case requires a different argument.

\section{$N$ can be squarefree}

For some of the possible shapes,  we were able to prove not only that
$N$ cannot be prime,  but more: it cannot be squarefree.  But for other shapes,
we were able to find squarefree tilings.   In this section we present
these tilings.

\begin{table}[ht]
\caption{Squarefree tilings}
\begin{center}
\begin{tabular}{lllll}
Shape &  tile & sides & picture & $N$  \\
\hline
$(2 \alpha, 2\beta, \alpha+\beta)$ & $(3,5,7)$ & (39,55,49) & Figure~\ref{figure:143} &  $143 = 11 \cdot 13$  \\
$(\alpha,2\beta,2\alpha+\beta)$ & $(7,8,13)$ & $(91,176,195)$ & Figure~\ref{figure:330} & $330 = 2 \cdot 3 \cdot 5 \cdot 11$  \\
$(\alpha,2\alpha,3\beta)$ &$(24,11,31)$  & $(961,1426,1155)$ & Figure~\ref{figure:4830} & $4830 = 2 \cdot 3 \cdot 5 \cdot 7 \cdot 23$ \\
\end{tabular}
\end{center}
\label{table:primes}
\end{table}

\begin{figure}[ht]
\begin{center}
\includegraphics[width=0.8\textwidth]{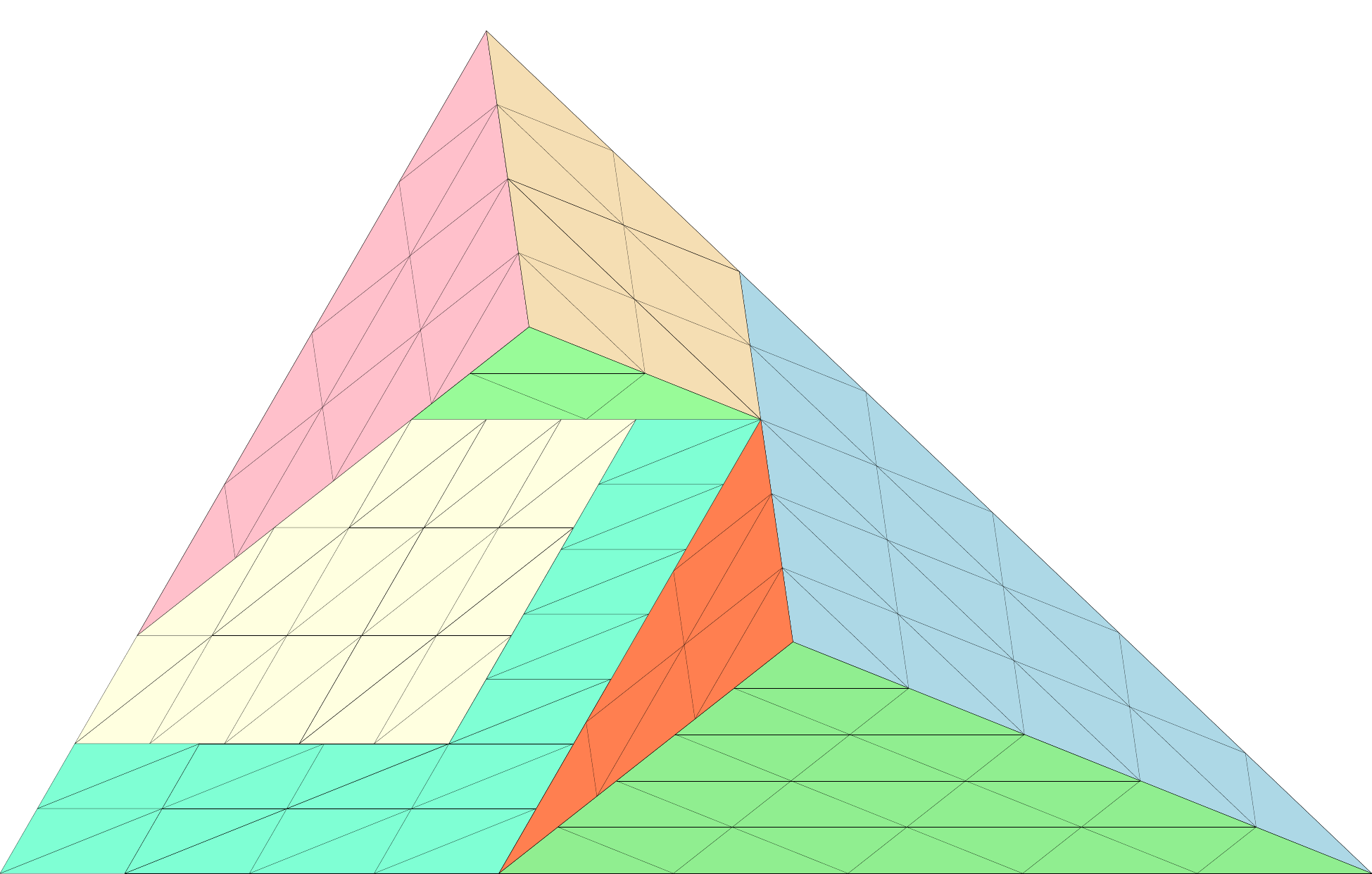}
\end{center}
\caption{$N = 143 = 11\cdot13$. Here $(3,5,7)$ tiles $(39,55,49)$.}
\label{figure:143}
\end{figure}

\begin{figure}[ht]
\begin{center}
\includegraphics[width=0.8\textwidth]{Figures/tiling330}
\end{center}
\caption{$N = 330 = 2\cdot3\cdot5\cdot11$. Here $(7,8,13)$ tiles $(91,176,195)$.}
\label{figure:330}
\end{figure}

\begin{figure}[ht]
\begin{center}
\includegraphics[width=\textwidth]{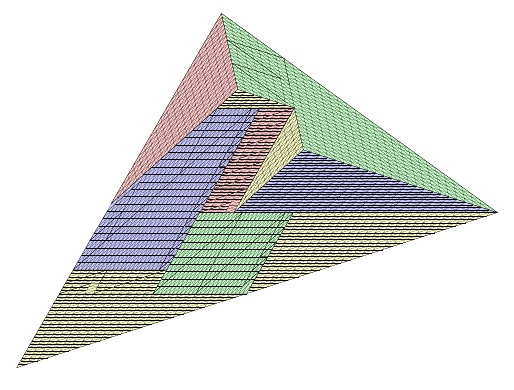}
\end{center}
\caption{$N = 4830$. Here $(24,11,31)$ tiles $(961,1426,1155)$. }
\label{figure:4830}
\end{figure}

\FloatBarrier
\section{Some lemmas valid for $3\alpha + 2\beta = \pi$}

\subsection{Angles}
Suppose we are constructing a tiling, and have to fill angle  $\alpha + \beta$.  That can only be done
with two tiles, each contributing one angle.   The following lemma gives the general fact of which this
is a special case.

\begin{lemma} \label{lemma:angles}  Suppose $\alpha$ is not a rational multiple
of $\pi$ and $3\alpha+2\beta = \pi$.  Then each angle of the form $p \alpha + q \beta$
can be expressed in that form for exactly one pair $(p,q)$ of nonnegative integers.
Hence if a sector of angle $p\alpha + q\beta$ is filled by tiles having
$x$ angles $\alpha$, $y$ angles $\beta$, and $z$ angles $\gamma$, then $x+2z = p$ and
$y + z = q$.  In particular
\begin{itemize}
\item A sector of angle $\alpha$ or $\beta$ is filled by one tile with that angle.
\item A sector of angle $\alpha+\beta$ is filled by two tiles with angles $\alpha$ and $\beta$.
\item A sector of angle $3\alpha$ is filled by three tiles with angle $\alpha$.
\item A sector of angle $\alpha + 2\beta$ is filled by three tiles with angles $\alpha,\beta,\beta$.
\end{itemize}
\end{lemma}

\begin{proof}   Suppose $p\alpha + q \beta = u\alpha + v\beta$.  Then
$(p-u) \alpha = (v-q)\beta$.  Unless $p=u$ and $v=q$, $\beta/\alpha$ will be
rational.  Then $3\alpha + 2\beta = \pi$  implies $\alpha$ and $\beta$ are
rational multiples of $\pi$.  To complete the proof, we still have to prove that
a sector with one of the itemized angles $\theta$ cannot have the form $u\gamma + v\alpha + w \beta$
with $u > 0$.  We have $\alpha < \gamma$ and $\beta < \gamma$ and $\alpha + \beta <\gamma$,
since $\gamma = 2\alpha + \beta$.  Suppose
$$3\alpha = \gamma + v\alpha + w \beta = (2+v)\alpha + (w+1)\beta.$$
Then $w+1 = 0$, which is impossible.  Suppose
$$\alpha + 2\beta = \gamma + v \alpha + w \beta =  (2+v)\alpha + (w+1)\beta.$$
Then $2+v = 1$, which is impossible.
\end{proof}

\begin{lemma}[angles at the vertices] \label{lemma:nogammacorner}
Suppose triangle $T$ is tiled by a tile with angles $(\alpha,\beta,\gamma)$,
where $3\alpha + 2\beta = \pi$ and $\alpha$ is not a rational multiple of $\pi$,
and $T$ is not similar to the tile.  Then no tile has its $\gamma$ angle at a vertex of
$T$.
\end{lemma}
\begin{proof}  Suppose a tile has its $\gamma$ angle at a vertex of $T$.  Then the angle of
$T$ there is $\gamma + r\alpha + s\beta$ for some nonnegative integers $r,s$, and the other
two angles of $T$ add up to $p\alpha + q\beta$, where $p+q \ge 2$, since each of them is
filled by at least one tile.  Since $\gamma = 2\alpha+\beta$, the angle sum of $T$ gives
$$ 3\alpha + 2\beta \ = \ \pi \ = \ (p+r+2)\alpha + (q+s+1)\beta, $$
so by Lemma~\ref{lemma:angles}, $p + r = 1$ and $q + s = 1$.  With $p+q \ge 2$ this forces
$p = q = 1$ and $r = s = 0$.  Then the angle of $T$ at that vertex is $\gamma$, and the other
two angles, being nonzero and adding up to $\alpha+\beta$, are $\alpha$ and $\beta$, again
by Lemma~\ref{lemma:angles}.  So $T$ is similar to the tile, contrary to hypothesis.
\end{proof}

\subsection{Parametrization}
We begin with the well-known parametrization of the side lengths of
a triangle with angles $(\alpha,\beta,\gamma)$ such that $3\alpha + 2\beta = \pi$.

\begin{lemma} \label{lemma:anglesOK}
Suppose $3\alpha + 2\beta = \pi$, and let $(a,b,c)$ be the sides opposite
$(\alpha,\beta,\gamma)$ in a triangle with those angles.  Then
$$ \gamma = \frac{\pi+\alpha} 2, \qquad \beta = \frac{\pi - 3\alpha} 2,$$
and
$$ \frac a c = 2 \sin \frac \alpha 2, \qquad
   \frac b c = 2\cos\alpha - 1 = 1 - \Big(\frac a c\Big)^2 .$$
Equivalently, $a^2 = c(c-b)$, and $\cos \alpha = (b+c)/(2c)$.
\end{lemma}

\begin{proof}
Doubling $\alpha+\beta+\gamma = \pi$ and subtracting $3\alpha + 2\beta = \pi$
gives $2\gamma - \alpha = \pi$, so $\gamma = (\pi+\alpha)/2$; then
$\beta = \pi - \alpha - \gamma = (\pi-3\alpha)/2$.

By the law of sines, $a/c = \sin\alpha / \sin\gamma$.  Now
$\sin \gamma = \sin(\pi/2 + \alpha/2) = \cos(\alpha/2)$, and
$\sin \alpha = 2 \sin(\alpha/2)\cos(\alpha/2)$, so $a/c = 2\sin(\alpha/2)$.
Likewise $\sin\beta = \sin(\pi/2 - 3\alpha/2) = \cos(3\alpha/2)$, so
$$ \frac b c = \frac{\cos(3\alpha/2)}{\cos(\alpha/2)}
 = 4\cos^2\frac\alpha 2 - 3 = 2(1+\cos\alpha) - 3 = 2\cos\alpha - 1,$$
using $\cos 3\theta = 4\cos^3\theta - 3\cos\theta$ with $\theta = \alpha/2$, and
$\cos\alpha = 2\cos^2(\alpha/2) - 1$.  Finally
$$ \Big(\frac a c\Big)^2 = 4 \sin^2 \frac\alpha 2 = 2(1-\cos\alpha),$$
so $1 - (a/c)^2 = 2\cos\alpha - 1 = b/c$.  Multiplying by $c^2$ gives
$bc = c^2 - a^2$, that is, $a^2 = c(c-b)$; and $b/c = 2\cos\alpha-1$ gives
$\cos\alpha = (b+c)/(2c)$.
\end{proof}

\begin{lemma} \label{lemma:gsquare}
Suppose $3\alpha + 2 \beta = \pi$, and $(a,b,c)$ are the sides
of a triangle with angles $(\alpha,\beta,\gamma)$. Suppose $(a,b,c)$
are integers with no common factor. Let $a/c = m/k$  in lowest terms.
Then
$$(a,b,c) = (mk,k^2-m^2,k^2),$$
and $k = \gcd(a,c)$ and $m = a/k$ and $m < k$.
\end{lemma}

\begin{proof}  Since $a$ and $c$ are positive we may take $k > 0$, and
then $m > 0$.  By Lemma~\ref{lemma:anglesOK},
$$\frac b c \ = \ 1 - \Big(\frac a c\Big)^2 \ = \ \frac{k^2-m^2}{k^2},$$
and therefore
$$(a,b,c) \ = \ \frac c {k^2}\,(mk,\ k^2-m^2,\ k^2).$$
The triple on the right has no common factor: a prime dividing both $k^2$
and $k^2-m^2$ would divide $k$ and $m$, contradicting $\gcd(m,k)=1$.
Since $(a,b,c)$ is likewise a triple of integers with no common factor,
and the two triples are positively proportional, they are equal.
Finally $\gcd(a,c) = \gcd(mk,k^2) = k \gcd(m,k) = k$, and $m = a/k$.
We have $m <k$ since $b > 0$ and $b = k^2-m^2$.
\end{proof}

\begin{lemma} \label{lemma:alphaoverpi}
Suppose $3\alpha + 2 \beta = \pi$, and $(a,b,c)$ are the sides
of a triangle with angles $(\alpha,\beta,\gamma)$. Suppose $(a,b,c)$
are integers with no common factor.   Then $\alpha/\pi$ and $\beta/\pi$
are irrational.
\end{lemma}

\begin{proof} Suppose $\alpha/\pi$ is rational.  Then so are $\beta/\pi$
and $\gamma/\pi$.  Then by the law of cosines,
$$ \cos \gamma = \frac {a^2 +b^2-c^2}{2ab} \in \Q,$$
and likewise $\cos \alpha$ and $\cos \beta$.  By Niven's theorem,
a rational multiple of $\pi$ in $(0,\pi)$ with rational cosine is
$\pi/3$, $\pi/2$, or $2\pi/3$.  The only way three such angles sum to $\pi$ is
$\tfrac \pi 3 + \tfrac \pi 3 + \tfrac \pi 3$.  Then the triangle is equilateral,
with $\alpha = \beta = \gamma = \pi/3$; but then $3\alpha + 2\beta = 5\pi/3 \neq \pi$.
So $\alpha/\pi$ is not rational, after all.
\end{proof}

\subsection{The relation lattice}

Throughout this subsection $m$ and $k$ are coprime with $0 < m < k$, and
$$ (a,b,c) \ = \ (mk,\ k^2-m^2,\ k^2),$$
as in Lemma~\ref{lemma:gsquare}.

\begin{definition}
The {\em relation lattice} of the tile is
$$ K \ = \ \{\,(x,y,z) \in {\mathbb Z}^3 \ :\ xa + yb + zc = 0\,\}.$$
\end{definition}

An identity among edge lengths of the kind that a line of a tiling can witness, say
$ha = pb + qc$ with $h,p,q \ge 0$, is exactly a point $(h,-p,-q)$ of $K$;  and similarly
for the $b$- and $c$-relations.  So the possible relations are the points of $K$ with one
coordinate positive and the other two non-positive, or the negatives of such points, and
the question of the least length of a relation of a given type is a question about $K$.

\begin{lemma} \label{lemma:lattice}
$K$ is a free abelian group of rank $2$, with basis
$$ e_1 = (k,\,0,\,-m), \qquad e_2 = (m,\,k,\,-k).$$
Consequently, for every $(x,y,z) \in K$,
$$ k \mid y \qquad\hbox{and}\qquad m \mid y+z .$$
\end{lemma}

\begin{proof}
Both $e_1$ and $e_2$ lie in $K$:  indeed $ka - mc = mk^2 - mk^2 = 0$ and
$ma+kb-kc = m^2k + k(k^2-m^2) - k^3 = 0$.

Conversely let $(x,y,z) \in K$.  Reducing $xa+yb+zc = 0$ modulo $k$ gives
$-ym^2 \equiv 0$, and $\gcd(m,k) = 1$, so $k \mid y$;  write $y = k\eta$.  Then
$(x,y,z) - \eta e_2 = (x-\eta m,\,0,\,z+\eta k)$ lies in $K$, so
$$ (x-\eta m)\,mk \ = \ -(z+\eta k)\,k^2, \qquad\hbox{that is}\qquad
   (x-\eta m)\,m \ = \ -(z+\eta k)\,k .$$
Hence $k \mid (x-\eta m)m$, and again $\gcd(m,k)=1$, so $k \mid x - \eta m$;  write
$x-\eta m = k\xi$, and then $z+\eta k = -\xi m$.  Thus
$(x,y,z) = \eta e_2 + \xi e_1$, so $e_1$ and $e_2$ generate $K$.  They are independent
over ${\mathbb Z}$, since the second coordinate of $\eta e_2 + \xi e_1$ is $k\eta$ and
the first is then $\eta m + k\xi$;  so $K$ is free of rank $2$ with that basis.

Finally, writing $(x,y,z) = \eta e_2 + \xi e_1$ we get $y = k\eta$ and
$y + z = k\eta + (-k\eta - m\xi) = -m\xi$, which gives the two divisibilities.
\end{proof}

\begin{theorem}[arithmetic of the relation lattice] \label{theorem:minlengths}
Let $m, k$ be coprime positive integers with $m < k$, and put
$$ a = mk, \qquad b = k^2 - m^2, \qquad c = k^2 . $$
For positive integers $h$ and nonnegative integers $p, q$, call an identity of one of
the nine forms below a \emph{relation of that type}, and require $p, q \ge 1$ in the
three mixed forms.  Then each type has solutions, and the least value of the common
quantity $ha$, $hb$ or $hc$ is:
$$
\begin{array}{ll}
\hbox{type} & \hbox{minimum length}\\[2pt]
\hline
ha = pb+qc & mk\bigl[k\lfloor k/m\rfloor + k - m\bigr]\\
ha = pb    & ab = mk(k^2-m^2)\\
ha = qc    & mk^2 = ka\\[2pt]
hb = pa+qc & kb = k(k^2-m^2)\\
hb = pa    & ab = mk(k^2-m^2)\\
hb = qc    & bc = k^2(k^2-m^2)\\[2pt]
hc = pa+qb & kc = k^3\\
hc = pa    & mc = mk^2\\
hc = qb    & bc = k^2(k^2-m^2)
\end{array}
$$
No hypothesis beyond $\gcd(m,k) = 1$ and $1 \le m < k$ is used;  in particular the
values are the same for $m = 1$.
\end{theorem}

\begin{proof}
Since $\gcd(m,k) = 1$ we have $\gcd(b,k) = \gcd(k^2-m^2,k) = \gcd(m^2,k) = 1$ and
$\gcd(b,m) = \gcd(k^2,m) = 1$, whence
$$ \gcd(a,b) = 1, \qquad \gcd(b,c) = \gcd(b,k^2) = 1, \qquad
   \gcd(a,c) = \gcd(mk,k^2) = k . $$
A relation $hX = pY$ in positive integers has least common value
$\operatorname{lcm}(X,Y) = XY/\gcd(X,Y)$, giving $ab$ for $\{a,b\}$, $bc$ for $\{b,c\}$,
and $ac/k = mk^2$ for $\{a,c\}$.  That settles the six pure types.

For the mixed types we use Lemma~\ref{lemma:lattice}:  every element of $K$ is
$\xi e_1 + \eta e_2$ for unique $\xi,\eta \in {\mathbb Z}$, that is
\begin{equation}
 (x,y,z) \ = \ \bigl(\,\xi k + \eta m,\ \ \eta k,\ \ -\xi m - \eta k\,\bigr).
 \label{eq:latticegeneral}
\end{equation}

For a relation of the type $ha = pb + qc$:  here $(x,y,z) = (h,-p,-q)$, so
\eqref{eq:latticegeneral} gives $p = -\eta k$, and $p \ge 1$ forces $\eta \le -1$;
write $P = -\eta \ge 1$, so that $p = Pk$ and
$$ h \ = \ \xi k - Pm, \qquad q \ = \ \xi m - Pk .$$
The condition $q \ge 1$ is $\xi > Pk/m$, so the least admissible $\xi$ is
$\lfloor Pk/m \rfloor + 1$, and then
$$ h \ = \ h(P) \ := \ k\bigl(\lfloor Pk/m \rfloor + 1\bigr) - Pm .$$
Since $\lfloor x+y\rfloor \ge \lfloor x\rfloor + \lfloor y\rfloor$,
$$ h(P+1) - h(P) \ = \ k\Bigl(\Bigl\lfloor \frac{(P+1)k}m\Bigr\rfloor
   - \Bigl\lfloor \frac{Pk}m\Bigr\rfloor\Bigr) - m
   \ \ge \ k\Bigl\lfloor \frac km \Bigr\rfloor - m \ \ge \ k-m \ > \ 0,$$
so $h(P)$ is strictly increasing and the minimum is at $P = 1$.  Writing
$k = m\lfloor k/m\rfloor + \rho$ with $0 \le \rho < m$ gives $q = m - \rho \ge 1$, so
$P = 1$ is admissible, and $h = k\lfloor k/m\rfloor + k - m$.

For a relation of the type $hb = pa + qc$:  here $(x,y,z) = (-p,h,-q)$, so
\eqref{eq:latticegeneral} gives $h = \eta k$, and $h \ge 1$ forces $\eta \ge 1$;  write
$H = \eta$, so that $h = Hk \ge k$, and
$$ p \ = \ -\xi k - Hm, \qquad q \ = \ \xi m + Hk .$$
Then $p \ge 1$ forces $\xi \le -1$;  write $R = -\xi \ge 1$, so that $p = Rk - Hm$ and
$q = Hk - Rm$.  Taking $H = R = 1$ gives $p = q = k-m \ge 1$, so the minimum is
$h = k$, of length $kb$.

For a relation of the type $hc = pa + qb$:  here $(x,y,z) = (-p,-q,h)$, so
\eqref{eq:latticegeneral} gives $q = -\eta k$, and $q \ge 1$ forces $\eta \le -1$;
write $t = -\eta \ge 1$ and $s = -\xi$, so that $q = tk$ and
$$ h \ = \ tk + sm, \qquad p \ = \ tm + sk .$$
Then $h + p = (t+s)(k+m)$, so $T := t+s$ is a positive integer, and writing $t = T-s$,
$$ h \ = \ Tk - s(k-m).$$
If $s \le -1$ then $h \ge k + (k-m) > k$.  If $s = 0$ then $h = Tk \ge k$, with equality
iff $t = 1$.  If $s \ge 1$ then $t \ge 1$ gives $T \ge s+1$, so
$h \ge (s+1)k - s(k-m) = k + sm > k$.  Hence $h \ge k$, attained at $t=1$, $s=0$, that
is $kc = ma + kb$, of length $kc$.
\end{proof}

\subsection{Edge decompositions}

\begin{definition} An {\em edge decomposition} of a length $\ell$ is a triple $(p,q,r)$ of nonnegative
integers with $pa + qb + rc = \ell$.
\end{definition}
Lemma~\ref{lemma:lattice} and Theorem~\ref{theorem:minlengths} have immediate consequences for the possible edge decompositions.
The following lemma spells them out.

\begin{lemma} \label{lemma:edgedecompositions} Let $(a,b,c) = (km,k^2-m^2,k^2)$ with $0 < m < k$ and $\gcd(m,k)=1$.
Let $1 \le n < k$.  Then

(A) The only edge decomposition of $na$ is $(n, 0, 0)$.

(B) The only edge decomposition of $nb$ is $(0, n, 0)$.

(C) Every edge decomposition $(p,q,r)$ of $nc$ has $q = 0$.

(D) The only edge decomposition $(p,q,r)$ of $m(b + 2c)$ with $r \ge k$ is $(0, m, 2m)$; there is none if
$2m < k$.

(E) No edge decomposition $(p,q,r)$ of $m(b + c)$ has $r \ge k$.
\end{lemma}

\begin{proof}
Throughout, $b = (k-m)(k+m) \ge k+m > k$ and $a = mk \ge k$ and $c = k^2 > k$.

(A) Let $pa+qb+rc = na$ and suppose $(p,q,r) \ne (n,0,0)$.  Then $q+r > 0$, so
$qb+rc > 0$, and $(n-p)a = qb+rc$ shows $h := n-p \ge 1$;  also $h \le n < k$.
Thus $ha = qb+rc$ is an $a$-relation.  By Theorem~\ref{theorem:minlengths} every
$a$-relation has $h \ge k$:  for $ha = pb+qc$ the length is at least
$mk[k\lfloor k/m\rfloor + k-m]$, so $h \ge k\lfloor k/m\rfloor + k - m \ge 2k-m > k$;
for $ha = pb$ the length is at least $ab$, so $h \ge b > k$;  and for $ha = qc$ the
length is at least $mk^2 = ka$, so $h \ge k$.  This contradicts $h < k$.

(B) Let $pa+qb+rc = nb$ and suppose $(p,q,r) \ne (0,n,0)$.  Then $p+r>0$, so
$(n-q)b = pa+rc > 0$ and $h := n-q$ satisfies $1 \le h \le n < k$, making $hb = pa+rc$
a $b$-relation.  Every $b$-relation has $h \ge k$:  the minimum lengths are $kb$,
$ab$ and $bc$, giving $h \ge k$, $h \ge a \ge k$ and $h \ge c > k$ respectively.
This contradicts $h<k$.

(C) Let $pa+qb+rc = nc$ with $q > 0$.  Then $pa+qb>0$, so $(n-r)c = pa+qb$ and
$h := n-r$ satisfies $1 \le h \le n < k$.  Since $q>0$ the relation $hc = pa+qb$ is of
type $hc = pa+qb$ or of type $hc = qb$, whose minimum lengths are $kc$ and $bc$;  hence
$h \ge k$ or $h \ge b > k$.  This contradicts $h < k$.

(D) Let $m(b+2c) = pa+qb+rc$ with $r \ge k$.  Then
$$ pa + (q-m)b + (r-2m)c \ = \ 0, $$
so $(p,\,q-m,\,r-2m) \in K$, and by Lemma~\ref{lemma:lattice} there are integers
$\xi,\eta$ with
$$ (p,\ q-m,\ r-2m) \ = \ \xi(k,0,-m) + \eta(m,k,-k)
   \ = \ (\xi k + \eta m,\ \ \eta k,\ \ -\xi m - \eta k), $$
that is
$$ p = \xi k + \eta m, \qquad q = m + \eta k, \qquad r = 2m - \xi m - \eta k. $$
From $q \ge 0$ we get $\eta \ge -m/k > -1$, so $\eta \ge 0$.

Suppose $\eta \ge 1$.  From $p \ge 0$ we get $\xi \ge -\eta m/k > -\eta$, so
$\xi \ge 1-\eta$, and therefore
$$ r \ = \ 2m - \xi m - \eta k \ \le \ 2m - (1-\eta)m - \eta k \ = \ m - \eta(k-m) \ < \ m \ < \ k, $$
contradicting $r \ge k$.  Hence $\eta = 0$.  Then $p = \xi k \ge 0$ gives $\xi \ge 0$,
and $q = m$, $r = (2-\xi)m$.  Since $r \ge k > m$ we need $2-\xi \ge 2$, so $\xi \le 0$
and therefore $\xi = 0$.  Thus $(p,q,r) = (0,m,2m)$, and such a decomposition exists
only if $2m = r \ge k$;  if $2m < k$ there is none.

(E) Let $m(b+c) = pa+qb+rc$ with $r \ge k$.  Then
$$pa+(q-m)b+(r-m)c = 0,$$
so as in (D),
$$ p = \xi k + \eta m, \qquad q = m + \eta k, \qquad r = m - \xi m - \eta k $$
for some integers $\xi,\eta$, and $q \ge 0$ gives $\eta \ge 0$.

If $\eta \ge 1$, then $p \ge 0$ gives $\xi \ge 1-\eta$ as in (D), whence
$$ r \ \le \ m - (1-\eta)m - \eta k \ = \ -\eta(k-m) \ < \ 0, $$
contradicting $r \ge k$.  So $\eta = 0$;  then $p = \xi k \ge 0$ gives $\xi \ge 0$ and
$$r = (1-\xi)m \le m < k,$$ again contradicting $r \ge k$.
\end{proof}

\begin{remark} In (C) one cannot conclude more.  The type $hc = pa$ has minimum length
$mc$ with $m < k$, and indeed $mc = ka$, so $(k,0,0)$ and $(0,0,m)$ are two edge
decompositions of $mc$.
\end{remark}

\subsection{Two $c$ edges on every side}

The following lemma is an improved version of Lemma~8 of \cite{beeson-triangletiling3}.
That lemma required that all the angles of $ABC$ be less than $\gamma$;  this improved
version only requires that no tile have a $\gamma$ angle at a vertex of $ABC$.
That weaker hypothesis applies to a tiling of the isosceles triangle $(\beta,\beta,3 \alpha)$,
in which $3\alpha$ might exceed $\gamma$, but there still cannot be a tile with a $\gamma$
angle at the $3\alpha$ vertex.  The proof of the improved version is essentially the same,
but still we repeat it here with minor modifications, so it can be checked directly.

\begin{lemma}[Two $c$ edges] \label{lemma:twocs}
Suppose triangle $ABC$  is
$N$-tiled by a tile with angles $(\alpha,\beta,\gamma)$.  Suppose
\smallskip

(i) $\gamma > \pi/2$, and
\smallskip

(ii) $\alpha/\beta$ is not rational, and
\smallskip

(iii) No tile has its $\gamma$ angle at a vertex of triangle $ABC$, and
\smallskip

(iv) One of the following two conditions holds:  Either
neither of $a,b$ is an integer multiple of the other, or else
neither of the vertices $A,C$ has two tiles with equal angles there
(that is, two $\alpha$ angles or two $\beta$ angles).
\smallskip

(v) the side $AC$ is not a single tile edge.
\smallskip

Then there
are at least two $c$ edges of tiles on side $AC$.
\end{lemma}

\begin{proof}[Proof of Lemma~\ref{lemma:twocs}]
We first show that at a boundary vertex $P$ on
the interior of a side of $ABC$, there is at most one tile with
its $\gamma$ angle at $P$.  The angles of the tiles at $P$ sum to $\pi$, so
they are $p\alpha+q\beta+r\gamma = \pi$ with $r \ge 1$.  Since
$\gamma > \pi/2$ by hypothesis~(i), we have $r = 1$; then
$p\alpha+q\beta = \pi-\gamma = \alpha+\beta$, and since $\alpha/\beta$ is
irrational by hypothesis~(ii), $p = q = 1$.  So exactly three tiles touch
$P$, contributing angles $\alpha$, $\beta$, and $\gamma$.

Next we prove that each side of $ABC$ has at least one $c$ edge.  Suppose
some side had none;  then every tile with an edge on that side has its $a$ or
$b$ edge there, and hence has a $\gamma$ angle at one of the two endpoints of
that edge, since the angles at the ends of an $a$ edge are $\beta$ and
$\gamma$ and at the ends of a $b$ edge are $\alpha$ and $\gamma$.  Let $n$ be
the number of such tiles;  their edges divide the side at $n-1$ interior
points.  None of those $n$ angles lies at an endpoint of the side, by
hypothesis~(iii);  and no point has two of them, since the angles at a
point interior to a side sum to $\pi$ while $2\gamma > \pi$ by
hypothesis~(i).  So $n$ angles occupy $n-1$ points, which is impossible.

The present lemma, however, claims more: there must be at least
two $c$ edges.  Suppose, to the contrary, that there is just one tile with
a $c$ edge on the side $AC$ of triangle $ABC$; call it Tile~1.  Then every
other tile with an edge on $AC$ has its $a$ or $b$ edge there, and so has a
$\gamma$ angle at one of the two endpoints of that edge.
We visualize $AC$ as horizontal with triangle $ABC$ above,
and use the words ``north'' and ``northwest'' accordingly.
See Figure~\ref{figure:twocs}.

\begin{figure}[ht]
\begin{center}
\includegraphics[width=0.8\textwidth]{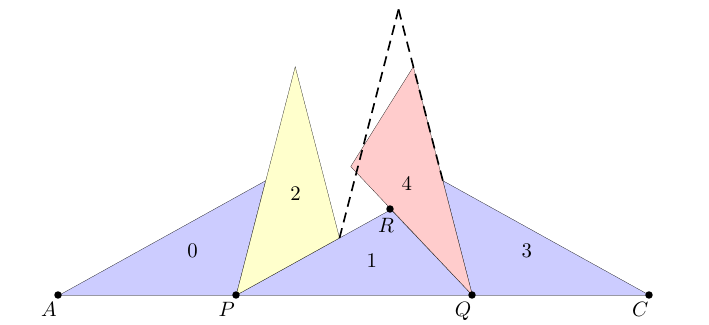}
\end{center}
\caption{Proof of Lemma~\ref{lemma:twocs}: another tile won't fit next to Tile~2}
\label{figure:twocs}
\end{figure}

\noindent
Since the hypotheses of the lemma remain true if (the names of)
$\alpha$ and $\beta$ are interchanged, we may assume without loss of
generality that $\alpha < \beta$.  Then by the law of sines, $a < b$.
Since $\gamma > \pi/2$ we also have $a < b < c$ (by the law of sines,
since $\beta < \pi - \gamma < \gamma$).

Let $PQ$ be the $c$ edge of Tile~1 lying on $AC$, and let $R$ be the
northern vertex of Tile~1.  Interchanging $P$ and $Q$ if necessary, suppose
that Tile~1 has its $\beta$ angle at $Q$.  Then Tile~1 has its $\alpha$
angle at $P$ and its $\gamma$ angle at $R$, the side $PR$ opposite $Q$
has length $b$, and the side $QR$ opposite $P$ has length $a$.
If $P$ is not a vertex of $ABC$, let Tile~0 be the tile on $AC$ immediately
west of $P$; if $Q$ is not a vertex of $ABC$, let Tile~3 be the tile on
$AC$ immediately east of $Q$.

We list some facts that will be used repeatedly.
\smallskip

(A) Tile~0 has its $\gamma$ angle at $P$, and Tile~3 has its $\gamma$ angle
at $Q$.  Indeed, let $n$ be the number of tiles with an edge on the segment
$AP$.  Their edges divide $AP$ at $n-1$ interior points.  Each of these
tiles has its $a$ or $b$ edge on $AP$, and so has a $\gamma$ angle at an
endpoint of its own edge; none has it at $A$, by hypothesis~(iii); and
since $2\gamma > \pi$, no point carries two $\gamma$ angles.  So $n$
angles are distributed among the $n-1$ interior points together with $P$,
one to a point, and in particular one lies at $P$.  The only one of these
tiles having $P$ as an endpoint of its edge is Tile~0.  The argument for
Tile~3 is the same.
\smallskip

(B) No tile has its $\gamma$ angle at $P$ except a tile with an edge on
$AC$.  Indeed, if $P$ is a vertex of $ABC$ this is hypothesis~(iii);
and if $P$ is not a vertex of $ABC$, then Tile~0 has its $\gamma$ angle at
$P$ by (A), and since $2\gamma > \pi$ there can be no second $\gamma$ angle
at $P$.
\smallskip

(C) The angle of $ABC$ at $P$ is greater than $\alpha$; consequently there
is a tile adjacent to $PR$ at $P$, which we call Tile~2.  Indeed, if $P$ is
not a vertex of $ABC$ the angle there is $\pi$.  Suppose then that $P$ is a
vertex of $ABC$ and its angle is $\alpha$.  Tile~1 fills that corner, so
$PR$ lies along a side of $ABC$, and $R$ lies on that side.  Then $Q$ and
$R$ both lie on the boundary of $ABC$; since $ABC$ is convex, no tile edge
along the line $QR$ extends beyond $Q$ or beyond $R$.  As $|QR| = a$ and
$a$ is the shortest side of the tile, exactly one tile lies on the far side
of $QR$ from Tile~1, and its $a$ edge is $QR$.  The two angles adjacent to
the side $a$ are $\beta$ and $\gamma$, so that tile has its $\gamma$ angle
at $Q$ or at $R$.  Not at $R$: the total angle at $R$ is $\pi$, since $R$
lies on a side of $ABC$, while Tile~1 contributes $\gamma$ there and
$2\gamma > \pi$.  Not at $Q$: if $Q$ is a vertex of $ABC$ that contradicts
hypothesis~(iii), and otherwise Tile~3 has its $\gamma$ angle at $Q$ by
(A), and again $2\gamma > \pi$.
\smallskip

(D) Suppose $R$ lies on a side of $ABC$, or $R$ lies strictly inside an
edge of a tile that does not lie along the line $PR$.
We will show that Tile~2 has its $a$ edge on $PR$, and so does every other
tile northwest of $PR$.
Indeed, under that supposition, $PR$ is terminated at $R$, so no tile
edge along $PR$ contains $R$ in its interior.  That is, every tile edge
along $PR$ is contained in $PR$.  Then no tile edge on $PR$ is a $c$ edge,
as $c > b = |PR|$.
If Tile~2 had its $b$ edge on $PR$, that edge would be all of $PR$, so its
$\gamma$ angle would be at $P$ or at $R$; not at $R$, since Tile~1 has its
$\gamma$ angle at $R$ and $2\gamma > \pi$, and not at $P$, by (B).  So
Tile~2 has its $a$ edge on $PR$, and each remaining edge along $PR$ is
shorter than $b$ and hence is an $a$ edge.  Thus, as claimed,
Tile~2 has its $a$ edge on $PR$, and so does every other
tile northwest of $PR$.
\smallskip

(E) The hypothesis of (D) holds as soon as $R$ lies on a side of $ABC$, or
some tile other than Tile~1 has an angle of $\gamma$ at $R$, or some tile
contains $R$ in the interior of one of its edges not lying along the line $PR$.
For a tile edge through
$R$ along $PR$ contributes a straight angle at $R$ on the northwest side,
and Tile~1 contributes $\gamma$; the total angle at $R$ is $\pi$ in the
first case and $2\pi$ otherwise, and $\pi + \gamma > \pi$, while
$\pi + \gamma + \pi > 2\pi$ and $\pi + \gamma + \gamma > 2\pi$, the last
because $2\gamma > \pi$.
\smallskip

We also state once a pigeon-hole argument that will be used several times.
Suppose $R$ lies on a side of $ABC$, or $R$ lies strictly inside an
edge of a tile that does not lie along the line $PR$, and
 every tile northwest of $PR$ has its $a$ edge on $PR$.   Let $n$
be the number of such tiles; since $|PR| = b > a$ we have $n \ge 2$, and
the edges divide $PR$ at $n-1$ interior points.  Each of these tiles has
its $a$ edge on $PR$, so each has a $\gamma$ angle at an endpoint of its
own edge.  None is at $P$, by (B); and none is at $R$, since Tile~1 has its
$\gamma$ angle at $R$ and $2\gamma > \pi$.  Since the angles at $R$ of the
tiles other than Tile~1 that have a vertex there sum to $\pi-\gamma < \gamma$,
as in (D), no
interior point of $PR$ has two $\gamma$ angles.  So $n$ angles must
occupy $n-1$ points, which is impossible.  We refer to this as {\em the
pigeon-hole argument}.
\smallskip

We first show that one of $P$, $Q$ is a vertex of $ABC$.  Assume, for
proof by contradiction, that neither is.  Then Tile~0 and Tile~3 both
exist.  By (A), Tile~3 has its $\gamma$ angle at $Q$; so by the fact about boundary
vertices recorded at the start of the proof, exactly three tiles touch $Q$,
contributing $\alpha$, $\beta$, and $\gamma$.  Tile~1 contributes the
$\beta$ and Tile~3 the $\gamma$, so a single tile, which we call Tile~4,
fills the remaining angle $\alpha$ between them.
 Then Tile~4 has its $b$ or $c$ edge along $QR$, since
its angle at $Q$ is $\alpha$.  As $|QR| = a$ and $a < b < c$, the edge of
Tile~4 along $QR$ extends past $R$, so $R$ lies in the interior of an edge
of Tile~4.  By (E) the hypothesis of (D) holds, so by (D) every tile
northwest of $PR$ has its $a$ edge on $PR$, and the pigeon-hole argument
gives a contradiction.  This proves that one of $P$, $Q$ is a vertex of
$ABC$.
\smallskip

Now we argue by cases.
\smallskip

Case 1: $Q$ is a vertex of $ABC$, i.e.\ $Q = C$.  If the angle of $ABC$
at $Q$ equals $\beta$, then $QR$ lies along the side $CB$ of $ABC$, and
there is no tile adjacent to Tile~1 at $Q$; call this Case~1a.
Otherwise let Tile~4 be the tile adjacent to Tile~1 at $Q$, sharing the
edge $QR$ with Tile~1.  Tile~4 cannot have its $\gamma$ angle at $Q$,
by hypothesis~(iii).  So either Tile~4 has its $\alpha$ angle at $Q$
(Case~1b) or its $\beta$ angle at $Q$ (Case~1c).  These three subcases
are exhaustive.
\smallskip

Case 1a: the angle of $ABC$ at $Q$ is $\beta$.  Then $QR$ lies along the
side $CB$, so $R$ lies on a side of $ABC$.  By (E) the hypothesis of (D)
holds, so by (D) every tile northwest of $PR$ has its $a$ edge on $PR$, and
the pigeon-hole argument gives a contradiction.  That completes Case~1a.
\smallskip

Case 1b: Tile~4 has its $\alpha$ angle at $Q$.  Then the two edges of
Tile~4 at $Q$ are its $b$ edge and its $c$ edge, and since
$|QR| = a < b < c$, the edge of Tile~4 along $QR$ extends past $R$, so $R$
lies in the interior of an edge of Tile~4.  By (E) the hypothesis of (D)
holds, and the pigeon-hole argument gives a contradiction as in Case~1a.
That completes Case~1b.
\smallskip

Case 1c: Tile~4 has its $\beta$ angle at $Q$.  Then Tile~1 and Tile~4
both have a $\beta$ angle at $Q$, so there are two equal angles at the
vertex $C$ of $ABC$; hence by hypothesis~(iv), $b$ is not an integer
multiple of $a$.  The two edges of Tile~4 at $Q$ are its $a$ edge and its
$c$ edge. If its $c$ edge lies along $QR$, it extends past $R$, so $R$
lies in the interior of an edge of Tile~4, and the argument of Case~1b applies.
Otherwise Tile~4 has its $a$ edge on $QR$ and its $\gamma$ angle at $R$.
By (E) no tile edge along the line $PR$ contains $R$ in its interior,
so the edges of the tiles northwest of $PR$ partition $PR$, and none of them
is a $c$ edge.  We will show that Tile~2 does not have its
$b$ edge on $PR$.  Suppose it does: then that edge is all of $PR$,
its $\gamma$ angle is at $R$ (not at $P$, by (B)) and its $\alpha$ angle is
at $P$, as is Tile~1's.  Then $P$ is a vertex of $ABC$:  at a point interior
to $AC$, Tile~0 has its $\gamma$ angle there, and the remaining angle
$\pi-\gamma - 2\alpha = \beta-\alpha$ is not a nonnegative combination of
$\alpha$ and $\beta$.  So $P=A$, $Q=C$, and $AC$ is the single $c$ edge $PQ$,
contrary to (v).  Therefore, as claimed, Tile~2 does not have its $b$
edge on $PR$.  Hence Tile~2 has its $a$ edge on $PR$.  Then every other edge
along $PR$ is shorter than $b$ and so is an $a$ edge, and $b$ is an integer
multiple of $a$, contrary to (iv).
That completes Case~1c, and with it Case~1.
\smallskip

Case 2: $P$ is a vertex of $ABC$, and $Q$ is not.  Then Tile~3 exists, and
as in the argument above, Tile~4 fills the open $\alpha$ angle between
Tile~1 and Tile~3, and its edge along $QR$ extends past $R$, so $R$ lies in
the interior of an edge of Tile~4.  Tile~2 exists by (C).  By (E) the
hypothesis of (D) holds, and the pigeon-hole argument gives a
contradiction.  That completes Case~2, and the proof of the lemma.
\end{proof}

\section{Runs of $c$ edges} \label{section:runs}

Throughout this section, triangle $T$ is tiled by a tile with angles
$(\alpha,\beta,\gamma)$, where $3\alpha+2\beta=\pi$, and integer sides
$$ (a,b,c) = (mk,\,k^2-m^2,\,k^2), \qquad 0<m<k,\quad \gcd(m,k)=1, $$
and we assume that no tile has its $\gamma$ angle at a vertex of $T$.
By Lemma~\ref{lemma:alphaoverpi}, $\alpha$ and $\beta$ are linearly independent
over $\Q$, and $\gamma = 2\alpha+\beta$.

\begin{definition} \label{definition:run}
Let $S$ be a side of $T$.  A {\em run} of $c$ edges of length $n$ on $S$ is a sequence
$T_1,\ldots,T_n$ of tiles, each having its $c$ edge on $S$, such that the $c$ edges of
$T_i$ and $T_{i+1}$ have a common endpoint for $1\le i<n$.  These edges partition a
segment $UV$ of $S$, with $U$ an endpoint of the $c$ edge of $T_1$;  we say the run is
{\em supported by} $UV$.  The run is {\em maximal} if no edge on $S$ adjacent to $UV$
is a $c$ edge.  The run is {\em stopped at $V$} if either $V$ is a vertex of $T$, or
the tile whose edge on $S$ lies beyond $V$ and ends at $V$ has its $\gamma$ angle at
$V$;  similarly for $U$.
\end{definition}

Suppose $T_1,\ldots,T_n$ is a run supported by $UV$, and $T_1$ has its $\alpha$ angle
at $U$.  We draw $S$ horizontally, with $U$ west of $V$ and $T$ to the north.  Let $W$
be the vertex of $T_1$ at its $\gamma$ angle, so that $UW$ is the $b$ edge of $T_1$ and
the angle $VUW$ is $\alpha$.  Let $\I$ and $\J$ be the unit vectors in the directions
of $UV$ and $UW$, and put
$$ x(u,v) \ = \ U + uc\,\I + vb\,\J . $$
For integers $i,j\ge 0$ let
$$ \Delta_{i,j} = x(i,j)\,x(i+1,j)\,x(i,j+1), \qquad
   \nabla_{i,j} = x(i+1,j)\,x(i,j+1)\,x(i+1,j+1). $$
Then $\Delta_{i,j}$ is a translate of $T_1$, with $\alpha,\beta,\gamma$ at
$x(i,j)$, $x(i+1,j)$, $x(i,j+1)$;  and $\nabla_{i,j}$ is the image of $\Delta_{i,j}$
under the half-turn about the midpoint of their common $a$ edge, with
$\gamma,\beta,\alpha$ at $x(i+1,j)$, $x(i,j+1)$, $x(i+1,j+1)$.  The $c$ edge of
$\nabla_{i,j}$ is $x(i,j+1)x(i+1,j+1)$ and its $b$ edge is $x(i+1,j)x(i+1,j+1)$.
For $1\le r\le s$ let
$$ \mathcal T(s,r) \ = \ \{\Delta_{i,j} : j<r,\ i+j\le s-1\}\ \cup\
   \{\nabla_{i,j} : j<r,\ i+j\le s-2\}, $$
which tiles the trapezoid $$U\,x(s,0)\,x(s-r,r)\,x(0,r)$$ (for $r=s$, the triangle
$U\,x(s,0)\,x(0,s)$).  We say $\mathcal T(s,r)$ {\em occurs} if each of its members is
a tile of the given tiling.  Note that $\Delta_{i,0}$ is a tile if and only if
$T_{i+1} = \Delta_{i,0}$, i.e.\ $T_{i+1}$ has its $\alpha$ angle at its western end.
Let $\ell_r$ be the line through $x(0,r)$ parallel to $S$.

\begin{lemma}[Trapezoid lemma] \label{lemma:block}
Let $T_1,\ldots,T_n$ be a run on a side $S$ of $T$, supported by $UV$, with $n<k$,
stopped at $V$, and with $T_1$ having its $\alpha$ angle at $U$.  Let
$1\le r<s\le n$.  Then the following cannot both hold:
\begin{enumerate}[{\rm (i)}]
\item $\mathcal T(s,r)$ occurs;
\item for some $i_0$ with $0\le i_0<s-r$, a tile north of $\ell_r$ has a vertex at
$x(i_0,r)$ and its $b$ edge on $\ell_r$, running east from $x(i_0,r)$.
\end{enumerate}
\end{lemma}

\begin{proof}
Suppose otherwise, and choose such $(s,r,i_0)$ with $s$ as large as possible.  Put
$R = x(s-r,r)$, $q_j = x(s-j,j)$ for $0\le j\le r$ (so $q_0 = x(s,0)$ and $q_r = R$),
and let $\sigma$ be the segment $q_0R$.  The west side of $\sigma$ consists of the $a$
edges of $\Delta_{s-1-j,j}$, $0\le j<r$;  the one on $\Delta_{s-1-j,j}$ runs from
$q_j$, where that tile has its $\beta$ angle, to $q_{j+1}$, where it has its $\gamma$
angle.

{\em $R$ is not a vertex of a tile north of $\ell_r$.}  South of $\ell_r$ the segment
from $x(0,r)$ to $R$ is the union of the $c$ edges of $\nabla_{i,r-1}$, $0\le i<s-r$;
so the edges on $\ell_r$ of the tiles north of $\ell_r$ cover it without overlap.  By
(ii), $x(i_0,r)$ is a vertex of one of those tiles and is followed by a $b$ edge.  If
$R$ were also a vertex of one of them, the edges between $x(i_0,r)$ and $R$ would be an
edge decomposition of $(s-r-i_0)c$ containing a $b$ edge;  as $1\le s-r-i_0<k$, this
contradicts Lemma~\ref{lemma:edgedecompositions}, part~C.  Hence $R$ lies in the interior of
an edge, on $\ell_r$, of a tile north of $\ell_r$.  In particular $R$ is an interior
point of $T$, and no tile edge on the line of $\sigma$ extends beyond $R$.

{\em The tile at $R$.}  South of $\ell_r$, the tiles $\nabla_{s-r-1,r-1}$ and
$\Delta_{s-r,r-1}$ have angles $\alpha$ and $\gamma$ at $R$, filling the sector between
the western part of $\ell_r$ and $\sigma$.  So the sector between $\sigma$ and the
eastern part of $\ell_r$ has angle $\pi-\alpha-\gamma = \beta$, and by Lemma~\ref{lemma:angles}, it is
filled by one tile $T_R$ with its $\beta$ angle at $R$;  one of its edges lies on
$\sigma$.

See Figure~\ref{figure:slant} for an illustration of the situation.
The point labeled $P$ is
$x(0,r)$.  $UPRV$ is the trapezoid for which the lemma is named. The proof
chooses the trapezoid with the smallest height,  regardless of width, with
a $b$ edge on the north boundary. The lower right corner might be $V$ or
a point west of $V$.

\begin{figure}[ht]
\includegraphics[width=\textwidth]{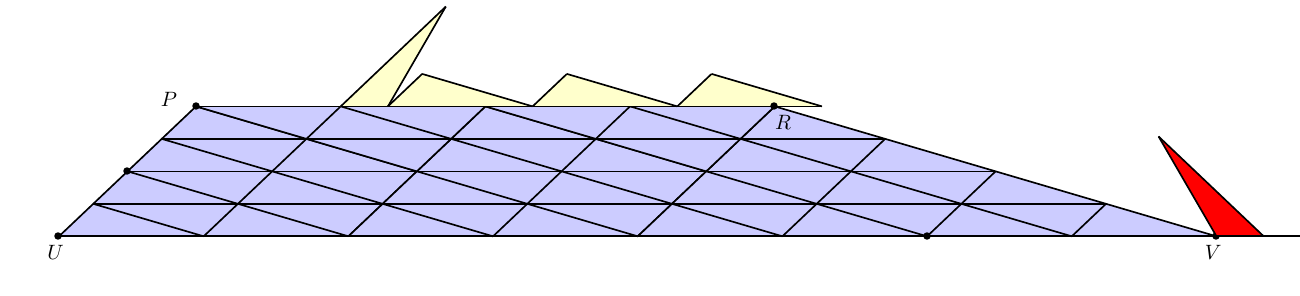}
\caption{The trapezoid lemma, when $s=n$ so $q_0 = V$. The yellow $b$ edge
forces a yellow edge to extend beyond $R$.
The eastern boundary $RV$ is too short for an $a$-relation,
and four $a$-edged tiles cannot fit because of the red stop and the overlap at $R$.
}
\label{figure:slant}
\end{figure}

{\em The lower end of $\sigma$.}  North of $S$, the tile $T_s = \Delta_{s-1,0}$ has its
$\beta$ angle at $q_0$, between $S$ and $\sigma$.
Suppose first that $s=n$ and $V$ is not a vertex of $T$.  Since the run is stopped at
$V$, the tile on $S$ east of $V$ has its $\gamma$ angle at $V$, so the sector at $V$
between it and $\sigma$ has angle $\pi-\beta-\gamma = \alpha$.  By Lemma~\ref{lemma:angles},
 it is filled by
a single tile with its $\alpha$ angle at $V$, whose edge on $\sigma$ is its $b$ or its
$c$ edge.  The edges on $\sigma$ of the tiles east of $\sigma$, from $T_R$ down to this
tile, form an edge decomposition of $ra$ containing a $b$ or $c$ edge;  as $r<k$ this
contradicts Lemma~\ref{lemma:edgedecompositions}, part A.

Suppose next that $s=n$ and $V$ is a vertex of $T$, with angle $\theta$.  If
$\theta=\beta$, then $\sigma$ lies along a side of $T$ and $R$ is on the boundary of
$T$, which we have excluded;  so $\theta>\beta$ and there are tiles east of $\sigma$
with a vertex at $V$.  If $s<n$, then $S$ continues past $q_0$, and again there are
tiles east of $\sigma$ with a vertex at $q_0$.

In both of these cases the edges of the tiles east of $\sigma$ partition $\sigma$, so by
Lemma~\ref{lemma:edgedecompositions}, part~A, they are $r$ edges of length $a$, with
endpoints exactly $q_0,\ldots,q_r$.  Each such tile has its $\beta$ and $\gamma$ angles
at the two ends of its $a$ edge.  The top one is $T_R$, with $\beta$ at $q_r$, hence
$\gamma$ at $q_{r-1}$.  For $0<j<r$ the point $q_j$ is interior to $\sigma$, so since $\gamma > \pi/2$,
at most one tile east of $\sigma$ has its $\gamma$ angle there;  by downward induction
on $j$, the tile on $q_{j}q_{j+1}$ has its $\gamma$ angle at $q_j$ and its $\beta$
angle at $q_{j+1}$.  So it is the half-turn image of $\Delta_{s-1-j,j}$, namely
$\nabla_{s-1-j,j}$.  In particular $\nabla_{s-1,0}$ has its $\gamma$ angle at $q_0$.
This situation is illustrated in Figure~\ref{figure:slant2}, in which the
tiles $\nabla_{s-1-j,j}$ east of $\sigma$ are shown in red.
\begin{figure}[ht]
\includegraphics[width=\textwidth]{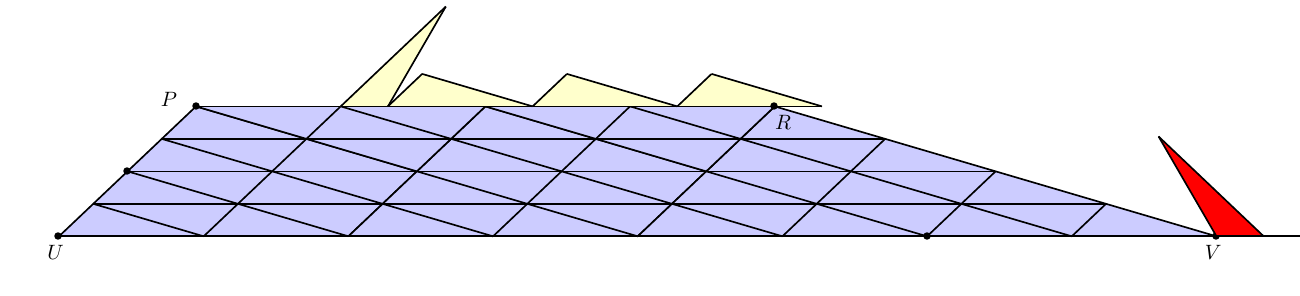}
\caption{The trapezoid lemma, when $s<n$. Now the right end is not stopped, but $Rq_0$
is still too short for an $a$ relation, so the tiles shown in red must exist.
If the green tiles all fit, then $s$ was not maximal, so there must be a pink tile.
But then also $s$ could have been larger, with smaller $r$.
}
\label{figure:slant2}
\end{figure}

If $s=n$, then $q_0=V$ is a vertex of $T$, and $\nabla_{s-1,0}$ has its
$\gamma$ angle at $V$,  contrary to our standing assumption.
Hence $s<n$.  At $q_0$, north of $S$, the tiles $T_s$ and $\nabla_{s-1,0}$ contribute
$\beta+\gamma$, so by Lemma~\ref{lemma:angles}, $T_{s+1}$ has its $\alpha$ angle at $q_0$, that is,
$T_{s+1} = \Delta_{s,0}$.  For $1\le j\le r-1$ the tiles $\Delta_{s-j-1,j}$,
$\Delta_{s-j,j-1}$, $\nabla_{s-j-1,j-1}$, $\nabla_{s-j-1,j}$, $\nabla_{s-j,j-1}$ have
angles $\beta,\gamma,\alpha,\gamma,\beta$ at $x(s-j,j)$, leaving a sector of angle
$\alpha$ between the $b$ edge of $\nabla_{s-j-1,j}$ and the $c$ edge of
$\nabla_{s-j,j-1}$.  By Lemma~\ref{lemma:angles}, it is filled by one tile $G_j$ with its $\alpha$ angle at
$x(s-j,j)$, whose sides there are its $b$ and $c$ edges.  So either
$G_j = \Delta_{s-j,j}$, or $G_j$ has its $b$ edge on $\ell_j$, running east from
$x(s-j,j)$.

If $G_j = \Delta_{s-j,j}$ for every $j$, then $\mathcal T(s+1,r)$ occurs, and (ii)
holds for $(s+1,r)$ with the same $i_0$, since $i_0 < s-r < (s+1)-r$.  Otherwise let
$j^*$ be the least $j$ with $G_j \ne \Delta_{s-j,j}$;  then $\mathcal T(s+1,j^*)$ occurs
and $G_{j^*}$ witnesses (ii) for $(s+1,j^*)$ with $i_0 = s-j^* < (s+1)-j^*$.  Either
way, since $s+1\le n$, this contradicts the maximality of $s$.
\end{proof}

\begin{lemma}[How not to start a run] \label{lemma:rootedblock}
Let $T_1,\ldots,T_n$ be a run on a side $S$ of $T$, supported by $UV$, with $n<k$,
stopped at $V$, and with $T_1$ having its $\alpha$ angle at $U$ and its $b$ side on $UW$.
  Then no tile with a
vertex at $U$, lying on the opposite side of the line $UW$ from $T_1$, has its $a$ edge
or its $c$ edge on the ray $UW$.
\end{lemma}

\begin{remark}
It is not assumed that the tiles of the run have the same orientation;  the proof shows
that, in the configuration being refuted, they would have to.  The   lemma  provides
a serious limitation on how a run can start.
\end{remark}

\begin{proof}[Proof of Lemma~\ref{lemma:rootedblock}]
Suppose $T_0$ is such a tile.  For $1\le w\le n$ let $H(w)$ say that
$\mathcal T(w,w)$ occurs, i.e.\ that the triangle $U\,x(w,0)\,x(0,w)$ is quadratically
tiled.  $H(1)$ holds, since $T_1 = \Delta_{0,0}$.  We show that $H(w)$ implies $w<n$
and $H(w+1)$;  then $H(n)$ yields $n<n$, a contradiction.

Assume $H(w)$, and put $P = x(0,w)$, $Q = x(w,0)$, $p_i = x(i,w-i)$ for $0\le i\le w$.

{\em The line $UP$.}  East of $UP$ lie the $b$ edges of $\Delta_{0,j}$, $j<w$.  West of
$UP$, the tiles have edges on the line $UP$ starting at $U$, the first being the $a$ or
$c$ edge of $T_0$.  If $P$ were a vertex of a tile west of $UP$, or a boundary point of
$T$, those edges would form an edge decomposition of $wb$ containing an $a$ or $c$ edge;
as $w<k$ this contradicts Lemma~\ref{lemma:edgedecompositions}, part~B.  Hence $P$ is
interior to $T$ and lies in the interior of an edge, on the line $UP$, of a tile west of
$UP$;  so that line continues as a tile edge north of $P$.

{\em The line $PQ$.}  West of $PQ$ lie the $a$ edges of $\Delta_{i,w-1-i}$, each with
its $\gamma$ angle at its end $p_i$ and its $\beta$ angle at $p_{i+1}$.  At $P$,
$\Delta_{0,w-1}$ has its $\gamma$ angle, so the sector between $PQ$ and the
continuation of $UP$ has angle $\pi-\gamma = \alpha+\beta$, filled by tiles with a
vertex at $P$, none with its $\gamma$ angle there since $\alpha+\beta<\gamma$.  At $Q$:
if $w<n$, the run continues past $Q$, so there are tiles east of $PQ$ with a vertex at
$Q$.  If $w=n$ and $V=Q$ is not a vertex of $T$, then as in the proof of
Lemma~\ref{lemma:block} the sector at $V$ adjacent to $PQ$ has angle $\alpha$ and is
filled by a tile whose edge on $PQ$ is a $b$ or $c$ edge.  If $w=n$ and $V$ is a vertex
of $T$, its angle is not $\beta$ (else $P$ would lie on a side of $T$), so there are
tiles east of $PQ$ with a vertex at $V$.  In every case the edges of the tiles east of
$PQ$ partition $PQ$, so by Lemma~\ref{lemma:edgedecompositions}, part~A, they are $w$ edges of
length $a$ with endpoints $p_0,\ldots,p_w$;  this excludes the case with a $b$ or $c$
edge.  No tile east of $PQ$ has its $\gamma$ angle at $p_0 = P$, and since $\gamma > \pi/2$,
at most one
has its $\gamma$ angle at each $p_i$, $0<i<w$;  so by induction on $i$ the tile on $p_ip_{i+1}$ has its
$\beta$ angle at $p_i$ and its $\gamma$ angle at $p_{i+1}$.  It is therefore the
half-turn image of $\Delta_{i,w-1-i}$, namely $\nabla_{i,w-1-i}$.  In particular
$\nabla_{w-1,0}$ has its $\gamma$ angle at $Q$.

The situation is illustrated in Figure~\ref{figure:runstart}.

\begin{figure}[ht]
\includegraphics[width=\textwidth]{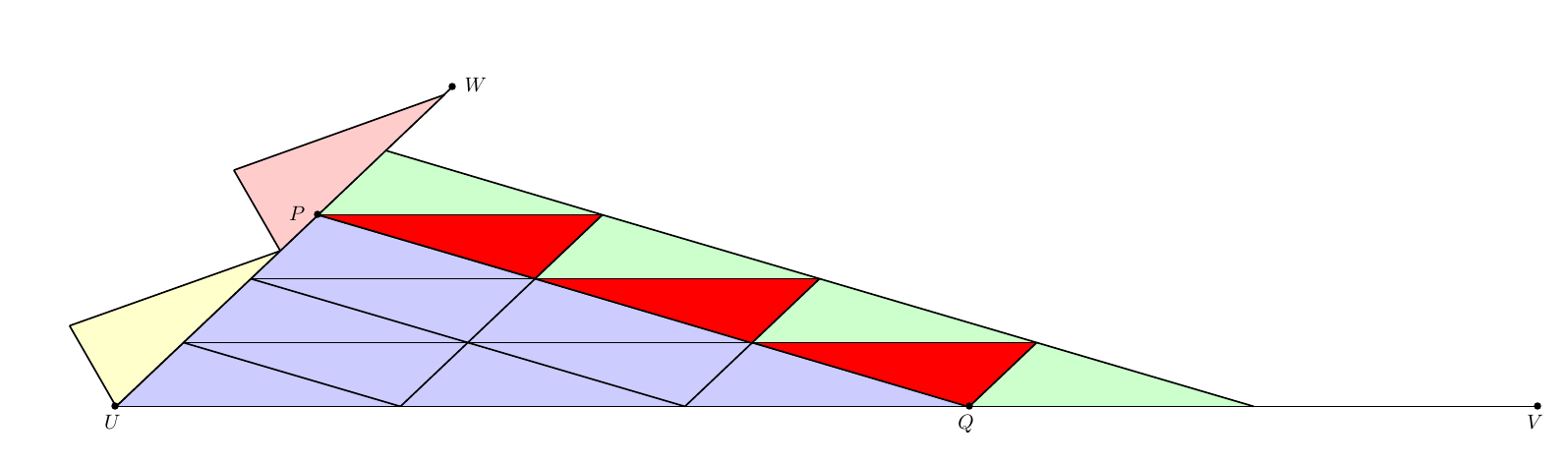}
\caption{$PU$ is too short to witness a $b$-relation, so $PQ$ is terminated at both ends.
Since $PQ$ is too short for an $a$-relation the red tiles exist.  Then the trapezoid
lemma implies the green tiles are oriented to extend the blue quadratic tiling one more column.
But eventually it will hit the stop at $V$.
}
\label{figure:runstart}
\end{figure}

Next we will show {\em $w<n$.}  If $w=n$, then $\nabla_{n-1,0}$ has its $\gamma$ angle at $V$.  If $V$ is
a vertex of $T$ this contradicts our standing assumption;  otherwise, since the run is
stopped at $V$, the tile on $S$ east of $V$ also has its $\gamma$ angle at $V$,
contradicting $\gamma > \pi/2$.  Therefore $w<n$, as claimed.

Finally we will show {\em $H(w+1)$.}  At $Q$, north of $S$, the tiles $T_w = \Delta_{w-1,0}$ and
$\nabla_{w-1,0}$ contribute $\beta+\gamma$, so $T_{w+1}$ has its $\alpha$ angle at $Q$
and $T_{w+1} = \Delta_{w,0}$.  For $1\le j\le w-1$, exactly as in the proof of
Lemma~\ref{lemma:block}, the five tiles $\Delta_{w-j-1,j}$, $\Delta_{w-j,j-1}$,
$\nabla_{w-j-1,j-1}$, $\nabla_{w-j-1,j}$, $\nabla_{w-j,j-1}$ leave a sector of angle
$\alpha$ at $x(w-j,j)$.  At $P$, the tiles $\Delta_{0,w-1}$ and $\nabla_{0,w-1}$
contribute $\gamma+\beta$ east of the line $UP$, leaving a sector of angle $\alpha$
between the continuation of $UP$ and the $c$ edge of $\nabla_{0,w-1}$.  So for
$1\le j\le w$ there is a tile $G_j$ with its $\alpha$ angle at $x(w-j,j)$, and either
$G_j = \Delta_{w-j,j}$ or $G_j$ has its $b$ edge on $\ell_j$, running east from
$x(w-j,j)$.  If $G_j = \Delta_{w-j,j}$ for all $j$, then $H(w+1)$ holds.  Otherwise, let
$j^*$ be the least $j$ with $G_j \ne \Delta_{w-j,j}$.  Then $\mathcal T(w+1,j^*)$
occurs, and $G_{j^*}$ satisfies (ii) of Lemma~\ref{lemma:block} with $s=w+1$, $r=j^*$,
$i_0 = w-j^*$;  since $1\le j^*\le w<w+1\le n$, Lemma~\ref{lemma:block} gives a
contradiction.
\end{proof}

\begin{lemma} \label{lemma:longrun}
Let $S$ be a side of $T$ such that whenever a tile $R$ has its $\alpha$ angle at an end
$V$ of $S$ and its $c$ edge on $S$, some other tile has a vertex at $V$ and its $a$ or
$c$ edge on the ray from $V$ along the $b$ edge of $R$.  Then $S$ supports a run of at
least $k$ consecutive $c$ edges.
\end{lemma}

\begin{proof}
Hypotheses~(i) and~(iii) of Lemma~\ref{lemma:twocs} hold:  $\gamma > \pi/2$  by
Lemma~\ref{lemma:anglesOK}, and no tile has its $\gamma$ angle at an
end of $S$ by the standing assumption of this section, that no tile has its $\gamma$
angle at a vertex of $T$.   By the second
paragraph of the proof of Lemma~\ref{lemma:twocs}, $S$ has at least one $c$ edge,
and therefore the $c$ edges of $S$ fall into $s\ge 1$ maximal runs.

Every edge on $S$ not belonging to these runs is an $a$ or $b$ edge, so its tile has its $\gamma$ angle at an
end of that edge;  since $\gamma > \pi/2$, no two of these $\gamma$ angles are at the same point, and
by Lemma~\ref{lemma:nogammacorner},
none is at an end of $S$.  Hence in a gap between two consecutive maximal runs, once an
edge has its $\gamma$ angle at its eastern end, so do all later edges of the gap;  so
either the first edge of the gap has its $\gamma$ angle at its western end, stopping the
run to the west, or the last has it at its eastern end, stopping the run to the east.
The part of $S$ before the first run is empty, so that the run begins at a vertex of
$T$, or else its first edge has its $\gamma$ angle at its eastern end, and then so does
its last edge;  either way the first run is stopped at its western end.  Likewise the
last run is stopped at its eastern end.  This gives at least $(s-1)+2>s$ distinct
stopped ends among $s$ runs, so some maximal run is stopped at both ends.

Let $X_0,\ldots,X_h$ be its vertices on $S$, $T_i$ the tile on $X_{i-1}X_i$, and suppose
$h<k$.

{\em Case 1: $T_1$ has its $\alpha$ angle at $X_0$.}  If $X_0$ is a vertex of $T$, the
hypothesis on $S$ provides a tile with a vertex at $X_0$ and its $a$ or $c$ edge on the
ray along the $b$ edge of $T_1$.  Otherwise, since the run is stopped at $X_0$, the tile
on $S$ west of $X_0$ has its $\gamma$ angle at $X_0$;  the remaining sector, of angle
$\pi-\alpha-\gamma = \beta$, is filled by one tile with its $\beta$ angle at $X_0$, by
Lemma~\ref{lemma:angles}, and its sides there are its $a$ and $c$ edges, one of which lies on the ray along
the $b$ edge of $T_1$.  Either way Lemma~\ref{lemma:rootedblock}, with $U = X_0$ and
$V = X_h$, is contradicted.

{\em Case 2: $T_h$ has its $\alpha$ angle at $X_h$.}  This is the mirror image of
Case~1.

{\em Case 3: $T_1$ has its $\beta$ angle at $X_0$ and $T_h$ has its $\beta$ angle at
$X_h$.}  Then $h\ge 2$, $T_1$ has its $\alpha$ angle at $X_1$, and $T_h$ has its
$\alpha$ angle at $X_{h-1}$.  Let $i$ be the largest index such that $T_i$ has its
$\alpha$ angle at $X_i$;  then $1\le i<h$, and $T_{i+1}$ has its $\alpha$ angle at
$X_i$.  At $X_i$, which is interior to $S$, the sector between the $b$ edges of $T_i$
and $T_{i+1}$ has angle $\pi-2\alpha = \alpha+2\beta$, so by Lemma~\ref{lemma:angles} it is filled by three
tiles with angles $\alpha,\beta,\beta$ at $X_i$, in some order.  At least one of the two
outer tiles has its $\beta$ angle at $X_i$, and its edge on the ray along the $b$ edge
of $T_i$ or of $T_{i+1}$ is its $a$ or its $c$ edge.  In the first case,
$T_i,\ldots,T_1$ is a run of length $i<k$ with $T_i$ having its $\alpha$ angle at
$U = X_i$, stopped at $V = X_0$;  in the second, $T_{i+1},\ldots,T_h$ is a run of length
$h-i<k$ with $T_{i+1}$ having its $\alpha$ angle at $U = X_i$, stopped at $V = X_h$.
Either way Lemma~\ref{lemma:rootedblock} is contradicted.

Hence $h\ge k$.
\end{proof}
\section{The shape $(\beta,\beta,3\alpha)$}
In this section, we assume that the isosceles triangle $ABC$ with base angles $\beta$ at $A$ and $C$,
and vertex angle $3\alpha$ at $B$, is $N$-tiled by a tile with angles $(\alpha,\beta,\gamma)$,
and corresponding sides $(a,b,c)$.  Our theorems will include the hypothesis that $\alpha$
is not a rational multiple of $\pi$.  This excludes the one commensurable case, $\alpha = \beta = \pi/5$,
in which $(\beta,\beta,3\alpha)$ is similar to the tile and has a reptiling.
Then  $\alpha/\beta$ is irrational, and by \cite[Theorem~9]{beeson-triangletiling3}
(or by Theorem~\ref{theorem:rationality})
the tile is rational; that is, the
ratios of the sides are rational.  We scale both $ABC$ and $(a,b,c)$ so that $a$, $b$, and $c$
are integers with no common factor.  Then there are positive coprime integers $m$ and $k$ with
$m<k$ such that
$$ (a,b,c) = (km,k^2-m^2,k^2).$$
These facts will not be repeated with every lemma and theorem in this section, but
they hold throughout.

In \cite[Theorem~19]{beeson-triangletiling3}, it is shown that for such tilings,
\begin{equation}
 N = j^2(3k^2-m^2) = j^2\left( 3c - \frac {a^2} c \right). \label{eq:1550}
\end{equation}

We call a tiling with $j=1$ in (\ref{eq:1550})  a {\em $j=1$ tiling}.

\begin{remark}
If the triangle of shape $(\beta,\beta,3\alpha)$ has an $N$-tiling with $N$ an odd prime,
then $j=1$ and $N = 3k^2-m^2$ with $m$ and $k$ coprime and $m<k$, and therefore
$N \equiv 11 \pmod{12}$.  Modulo $3$ we have $N \equiv -m^2$, which is $2$ unless
$3 \mid m$;  and if $m = 3m'$ then $3 \mid N$, so $N=3$ and $k^2-3m'^2 = 1$, giving
$k < 3m' = m$, contrary to $m<k$.  Modulo $4$, $N$ odd forces $m$ and $k$ of opposite
parity, and in either case $N \equiv 3$.  Conversely every prime $\equiv 11 \pmod{12}$ is
so represented, so the area equation alone excludes no such prime.   We give the
proof that every  prime $\equiv 11 \pmod{12}$ is represented.  For such a prime $p$,
$3$ is a quadratic residue modulo $p$ (reciprocity, since $p \equiv 3 \pmod 4$ and $p \equiv 2 \pmod 3$),
so $p$ is represented by one of the two forms $x^2-3y^2$ and $3y^2-x^2$ of discriminant 12.
The first represents no number $\equiv 2 \pmod 3$, so $p = 3k^2-m^2$; if $k\le m$, the
substitution $(m,k)\mapsto (\vert 2m-3k\vert,2k-m)$ preserves $3k^2-m^2$ and decreases $k$,
so a representation with $0 < m <k$ exists, and $\gcd(m,k)^2 | p$.
\smallskip

This explains why
geometry will be required to solve this case.
\end{remark}

\begin{lemma}[Isosceles-$\beta$ side lengths] \label{lemma:isobetasides}
Let $ABC$ be the isosceles triangle with angles $(\beta,\beta,3\alpha)$, with $X$ the
length of each of its two equal sides and $Y$ the length of its base.  Suppose $ABC$
is $N$-tiled by $(a,b,c)$.   Then
\begin{eqnarray*}
X &=& jk^3 = jkc \\
Y &=& j\,m\,(3k^2-m^2) \;=\; jm\,(b+2c) \\
 \frac{Y}{X} &=& \frac{a\,(b+2c)}{c^2}, \mbox{\quad and} \\
 Y &<& 2X
\end{eqnarray*}
\end{lemma}

\begin{proof}
First, we will prove $Y < 2X$.
 Drop the perpendicular from $B$ to $AC$.
  Since $ABC$ is isosceles the foot is the midpoint of the base, so in the
resulting right triangle the angle at a base corner is $\beta$, the adjacent leg is $Y/2$
and the hypotenuse is $X$.  Hence
$$ Y \;=\; 2X\cos\beta \ < \ 2X. $$
By the law of cosines in the tile,
$$\cos\beta = \frac {a^2+c^2-b^2}{2ac},$$
 and with
$$(a,b,c) = (mk,\,k^2-m^2,\,k^2),$$
we have
$$ a^2+c^2-b^2 \;=\; m^2k^2 + k^4 - (k^2-m^2)^2 \;=\; 3m^2k^2 - m^4 \;=\; m^2(3k^2-m^2),
 $$
 and $ ac \;=\; mk^3$,
so
\begin{eqnarray}
 \frac{Y}{X} &=& \frac{a^2+c^2-b^2}{ac} \;=\; \frac{m(3k^2-m^2)}{k^3}  \nonumber\\
 \frac{Y}{X}  &=& \frac{a(b+2c)} {c^2}.  \label{eq:1595}
 \end{eqnarray}
The area of $ABC$ is $\onehalf XY \sin \beta$;  the area of the tile is $\onehalf ac \sin \beta$.  Since $ABC$
is $N$-tiled, we have the area equation
$$ N = \frac{XY}{ac} = \frac{X^2}{ac}\cdot\frac{Y}{X} .$$
Substituting for $Y/X$ from (\ref{eq:1595}), we have
\begin{eqnarray}
 N \;=\; \frac{X^2}{ac}\cdot\frac{a\,(b+2c)}{c^2} \;=\; \frac{X^2\,(b+2c)}{c^3} .
   \label{eq:1599}
\end{eqnarray}

 Since $c = k^2$ and $b + 2c = 3k^2-m^2$, equation
\eqref{eq:1599} implies
 $$N = X^2(3k^2-m^2)/k^6.$$
Comparing that to \eqref{eq:1550}, we have
$$ X = jk^3 = jkc.$$
Then by \eqref{eq:1595},
$$ Y \;=\; X\cdot\frac{a(b+2c)}{c^2} \;=\; jk^3\cdot\frac{mk\,(3k^2-m^2)}{k^4}
   \;=\; jm(3k^2-m^2) \;=\; jm(b+2c) .  \qedhere
$$
\end{proof}

\begin{lemma}[no $b$-edges on sides] \label{lemma:isoscelesnob}
Suppose the isosceles triangle $ABC$ of shape $(\beta,\beta,3\alpha)$ is
$N$-tiled by $(a,b,c)$ as above, with $j=1$.  Then no tiles with $b$ edges
are supported on either of the two equal sides $AB$ and $BC$.
\end{lemma}

\begin{proof} By Lemma~\ref{lemma:isobetasides}, the length $X$ of the equal sides
is equal to $k^3 = kc$.  Let $X = kc =  pa + qb + rc = pkm + q(k^2-m^2) + rk^2$.
By Lemma~\ref{lemma:alphaoverpi}, $\alpha/\pi$ is irrational,  which is needed
to apply Lemma~\ref{lemma:twocs}.
We have $r \ge 1$ by the second paragraph of the proof of Lemma~\ref{lemma:twocs}
(every side has a $c$ edge).   Then $(k-r)c = pa + qb$.  If $q > 0$ this
is a relation of type $hc = pa + qb$ or $hc=qb$ with $h = k-r$.  Here $h \ge 1$,
since $k=r$ implies $p=q=0$.  Note that Theorem~\ref{theorem:minlengths} applies to
any relation whatever, not just to relations witnessed in the tiling.

By Theorem~\ref{theorem:minlengths}, the minimum lengths of relations of those types
are $kc$ and $bc$.
In either case the length exceeds $kc$,
since $b = k^2-m^2 = (k+m)(k-m) > k$.   Therefore $(k-r)c < kc$.
That contradiction came from the assumption  $q > 0$.  Therefore $q=0$.
Since $q$ is the number of $b$ edges of tiles supported on one of the equal sides
of $ABC$, that completes the proof.
\end{proof}

\begin{theorem} \label{theorem:isoscelesbeta}
There is no $j=1$ tiling of the isosceles-$\beta$ triangle
when $3\alpha+2\beta = \pi$ with $\alpha$ not a rational
multiple of $\pi$.   In particular, the number of tiles
 $N$ cannot be prime, or even squarefree.
\end{theorem}

\begin{proof}  Suppose there is such a tiling of an isosceles triangle
$ABC$, with base angles $\beta$ at $A$ and $C$, and vertex angle $3\alpha$ at $B$.
By Lemma~\ref{lemma:gsquare}, we can assume the sides of the tile are integers $(a,b,c)$
given by
$$ (a,b,c) = (mk, k^2-m^2, k^2) \mbox{\qquad where $\gcd(m,k)= 1$ and $0 < m < k$}.$$
According to Lemma~\ref{lemma:isobetasides}, the length of the two equal sides is $X = kc$,
and the length of the base is $Y = m(b+2c)$.

No tile has its $\gamma$ angle at a vertex of $ABC$ by Lemma~\ref{lemma:nogammacorner}, since $ABC$
is not similar to the tile.  That is the standing assumption of Section~\ref{section:runs}; therefore
the results of Section~\ref{section:runs} can be applied here.

We claim there is no $a$ edge on the base $AC$.   We apply Lemma~\ref{lemma:longrun} to base
$AC$.  By Lemma~\ref{lemma:angles}, the tiles supported by the base at $A$ and $C$
have their $\beta$ angles there, and hence cannot have their $\alpha$ angles
at $A$ or $C$, so the hypothesis of Lemma~\ref{lemma:longrun}  is satisfied vacuously.  Therefore
the base $AC$ supports a run of $k$ consecutive $c$ edges.  By Lemma~\ref{lemma:edgedecompositions},
part~D, the only possible edge decomposition of $AC$ is $(0,m,2m)$. (If $2m <k$, there is no
such decomposition, and the proof ends here.) In particular, the base supports no tile with
an $a$ edge on the base, as claimed.

Next we claim that one leg consists exactly of $k$ edges of length $c$.
By Lemma~\ref{lemma:isoscelesnob}, there are no $b$ edges on $AB$ or $BC$.
By Lemma~\ref{lemma:angles}, there are three tiles at vertex $B$, each with angle $\alpha$ at $B$.
Then the outer two have their $c$ edges on $AB$ and $BC$, and their $b$ edges against the
middle tile.   Without loss of generality, we can assume the $c$ edge of the middle tile is
at the west, against the $b$ edge of the western tile at $B$. At $A$ the hypothesis of Lemma~\ref{lemma:longrun}
is vacuous, as on $AC$.
Now we can apply Lemma~\ref{lemma:longrun}
to the side $AB$ of triangle $ABC$, since the tile at $B$ adjacent to $AB$
has its $\alpha$ angle at $B$ and the adjacent
tile has a $c$ edge against its $b$ edge.   Then by Lemma~\ref{lemma:longrun}, $AB$ supports
a run of at least $k$ consecutive $c$ edges.   Since the length of $AB$ is exactly $kc$, that means that
the edge decomposition is $(0,0,k)$:  $AB$ consists entirely of $c$ edges.  By Lemma~\ref{lemma:angles},
the tile at $A$
has its $\beta$ angle at $A$, so its $a$ edge must lie on $AB$ or $AC$; but both have been
ruled out, since $AC$ has no $a$ edges, and $AB$ consists only of $c$ edges, so it has no
$a$ edges either.
\end{proof}

\section{The triquadratic shape $(2\alpha, \beta, \alpha+\beta)$}

In this section we study tilings of
the triquadratic triangles.
In this entire section, we assume that $3\alpha + 2\beta = \pi$ and $\alpha$ is not
a rational multiple of $\pi$, and that a triangle $T$ of
shape $(2\alpha, \beta, \alpha+\beta)$ is $N$-tiled by a tile with angles $(\alpha,\beta,\gamma)$
and sides $(a,b,c)$ opposite those angles.  By Theorem~\ref{theorem:rationality},
the sides are commensurable; we assume $T$ and the tile have
been scaled so that $a$, $b$, and $c$ are integers with no common factor.

\subsection{Preliminaries}
According to Theorem~10 of \cite{beeson-triangletiling3}, if there is an $N$-tiling of a
triangle of shape $(2\alpha, \beta, \alpha+\beta)$,  then there exist integers $K$ and $M$
such that $N = 2K^2-M^2$,  and $s = 2 \sin(\alpha/2) = M/K$.  Here $M/K$ is not necessarily
in lowest terms.  We review Lemma~47 of \cite{beeson-triangletiling3},  where it is observed that
if $j = \gcd(M,K)$, so that $(M,K) = j(m,k)$ where $m/k$ is in lowest terms, then
$$ N = j^2(2k^2-m^2).$$
Hence, if $N$ is prime, or even squarefree, we must have $j=1$, since $0< m < k$ implies $2k^2-m^2 > 0$.
  We prove in Theorem~\ref{theorem:triquadratic} that there are no tilings with $j=1$;
  hence none with $N$ prime, or even squarefree.

The numbers $m$ and $k$ are intimately connected with the geometry of the tile,
as we showed in Lemma~\ref{lemma:gsquare}.

\begin{lemma}\label{lemma:gammagreaterpiover2}
Suppose $3\alpha + 2\beta = \pi$.  Then $\gamma > \pi/2$.
\end{lemma}

\begin{proof}
\begin{eqnarray*}
\pi &=& 3\alpha + 2\beta \\
    &=& \alpha + 2(\alpha + \beta) \\
    &=& \alpha + 2 (\pi - \gamma)
\end{eqnarray*}
Hence $2\gamma = \pi + \alpha > \pi$, since $\alpha > 0$, so $\gamma > \pi/2$.
\end{proof}

\begin{lemma}[Triquadratic side lengths]\label{lemma:sidelengths}
Let $ABC$ have the shape $(2\alpha, \beta,\alpha+\beta)$, with $2\alpha$ at $A$,
$\beta$ at $B$ and $\alpha+\beta$ at $C$, and $\alpha$ not a rational multiple of $\pi$.
 Let the tile have sides
$(a,b,c) = (km,k^2-m^2,k^2)$ with $\gcd(k,m) = 1$.  Let $ABC$ be $N$-tiled, and
$N = 2K^2-M^2$, with $(M,K) = j(m,k)$.  Then
$$ \vert AB \vert = Kc, \qquad \vert AC \vert = Kb, \qquad
   \vert BC \vert = M(b+c). $$
\end{lemma}

\begin{proof}
Let $d$ be the circumdiameter of the tile, so $a = d \sin\alpha$,
$b = d\sin\beta$, $c = d \sin\gamma$.  Since $\alpha+\beta = \pi - \gamma$,
$$ \sin(\alpha+\beta) = \sin\gamma = c/d .$$
By Lemma~\ref{lemma:anglesOK}, $\cos\alpha = (b+c)/(2c)$, so
$$ \sin 2\alpha \ = \ 2 \sin\alpha \cos \alpha \ = \
   \frac{2a} d \cdot \frac{b+c}{2c} \ = \ \frac{a(b+c)}{dc}.$$
By the law of sines in $ABC$, whose angles at $A$, $B$, $C$ are $2\alpha$,
$\beta$, $\alpha+\beta$,
$$ \vert BC \vert : \vert AC \vert : \vert AB \vert \ = \
   \sin 2\alpha : \sin\beta : \sin(\alpha+\beta) \ = \ \frac{a(b+c)} c : b : c .$$
Since $a/c = m/k$, multiplying through by $k$ gives
$$ \vert BC \vert : \vert AC \vert : \vert AB \vert \ = \ m(b+c) : kb : kc ,$$
so for some $\lambda > 0$,
$$ \vert BC \vert = \lambda m(b+c), \qquad \vert AC \vert = \lambda k b, \qquad
   \vert AB \vert = \lambda k c .$$

It remains to show $\lambda = j$.  The area of $ABC$ is
$$ \tfrac 1 2 \vert AB \vert \, \vert AC \vert \sin 2\alpha \ = \
   \tfrac 1 2 \lambda^2 k^2 bc \cdot \frac{a(b+c)}{dc} \ = \
   \frac{\lambda^2 k^2 a b (b+c)}{2d},$$
while the area of the tile is $\tfrac 12 ab \sin \gamma = abc/(2d)$.  Dividing,
and using $b+c = 2k^2-m^2$ and $c = k^2$,
$$ N \ = \ \frac{\lambda^2 k^2 (b+c)} c \ = \ \lambda^2(2k^2-m^2) .$$
Comparing with $N = j^2(2k^2-m^2)$ gives $\lambda = j$, and then
$\lambda k = K$ and $\lambda m = M$.
\end{proof}

\subsection{The trisector decomposition}
One observes that a known isosceles-$\beta$ tiling is constructed by tacking a quadratic tiling
onto a triquadratic tiling.  The $j=2$ case results in a 44-tiling of an isosceles triangle, as
illustrated in Figure~\ref{figure:44}.  The following lemma does not assume either triangle
is tiled, but just describes the geometry of the two triangles, which could be viewed as
a decomposition of the isosceles triangle into a triangle similar to the tile and a triquadratic triangle,
or as a construction of an isosceles triangle by adding something to a triquadratic triangle.
\begin{figure}[ht]
\begin{center}
\includegraphics[width=0.8\textwidth]{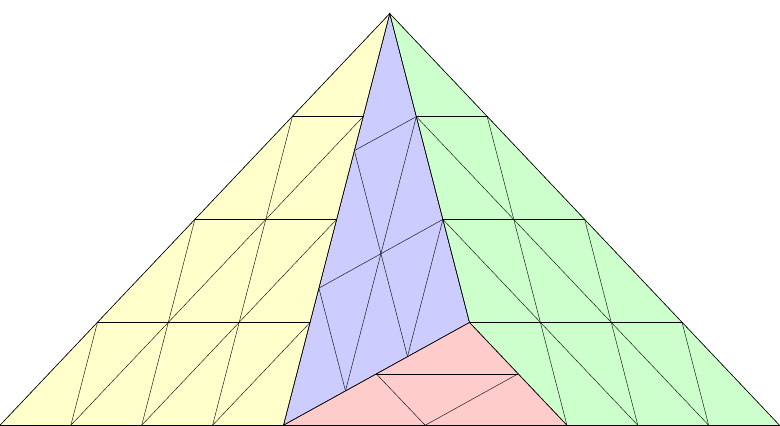}
\end{center}
\caption{The $44$-tiling of the isosceles-$\beta$ triangle with $(m,k) = (1,2)$,
$j = 2$:  the trisector cuts it into a quadratic tiling of
$16$ tiles and a triquadratic tiling of $28$.}
\label{figure:44}
\end{figure}

\begin{lemma}[the trisector decomposition] \label{lemma:trisector}
Let $ABC$ be the isosceles-$\beta$ triangle with base angles $\beta$ at $A$ and $C$ and
apex angle $3\alpha$ at $B$.  Let $(a,b,c)$ be coprime integers forming the sides of
a triangle with angles $(\alpha,\beta,\gamma)$.  Suppose
$\vert AB\vert = \vert CB \vert = jkc$.  Then
$\vert AC\vert = jm(b+2c)$.  Let $E$ be the point
where the trisector of the angle at $B$ nearer $A$ meets $AC$.  Then
\begin{enumerate}[{\rm (i)}]
\item triangle $ABE$ is similar to the tile, with ratio of similarity $jk$;  in
particular $\vert AE\vert = jka$ and $\vert BE\vert = jkb$;
\item triangle $EBC$ has the triquadratic shape, with $2\alpha$ at $B$, $\beta$ at $C$
and $\alpha+\beta$ at $E$, and its sides are given by
$$ \vert BC \vert = jkc,\  \vert BE \vert = jkb, \mbox{\ and \ } \vert EC \vert = jm(b+c).$$
\item $\vert EC\vert = jm(b+c)$, and $\vert AE\vert + \vert EC\vert = \vert AC\vert$.
\end{enumerate}
The same holds for the other trisector, by the reflection interchanging $A$ and $C$.
\end{lemma}

\begin{proof}
By Lemma~\ref{lemma:gsquare}, $(a,b,c) = (km,k^2-m^2,k^2)$, with $m<k$ and $\gcd(m,k)=1$.
The length $\vert AC\vert = jm(b+2c)$ follows from the ratio $Y/X$ calculated in the proof
of Lemma~\ref{lemma:isobetasides}, by trigonometry that did not depend on the assumption that the triangle is tiled.
Triangle $ABE$ has $\beta$ at $A$ and $\alpha$ at $B$, hence
$\pi-\alpha-\beta = \gamma$ at $E$, so it is similar to the tile;  the side $AB$ is
opposite $\gamma$, so the ratio of similarity is $\vert AB\vert/c = jk$, and the sides
opposite $\alpha$ and $\beta$ are $\vert AE\vert = jka$ and $\vert BE\vert = jkb$.
That proves~(i).

Triangle $EBC$ has $2\alpha$ at $B$, $\beta$ at $C$ and $\pi-\gamma = \alpha+\beta$ at
$E$, so it is triquadratic.  The side lengths in~(ii) follow by the law of sines.

To prove~(iii), we have
$$ \vert AE\vert + \vert EC\vert \ = \ jmk^2 + jm(b+c) \ = \ jm(b+2c)
   \ = \ \vert AC\vert . \qedhere$$
\end{proof}

\subsection{Tilings with $j=1$ of $(2\alpha,\beta,\alpha+\beta)$}

\begin{remark}
If the triangle of the triquadratic shape $(2\alpha,\beta,\alpha+\beta)$ has an $N$-tiling
with $N$ an odd prime, then by Theorem~10 of \cite{beeson-triangletiling3} we have
$N = j^2(2k^2-m^2)$, so $j=1$ and $N = 2k^2-m^2$ with $m$ and $k$ coprime and $m<k$;  and
therefore
$$ N \ \equiv \ \pm 1 \pmod 8 .$$
For $N$ is odd, so $m$ is odd and $m^2 \equiv 1 \pmod 8$.  If $k$ is even then
$4 \mid k^2$, so $2k^2 \equiv 0 \pmod 8$ and $N \equiv -1$;  and if $k$ is odd then
$k^2 \equiv 1 \pmod 8$, so $2k^2 \equiv 2$ and $N \equiv 1$.  Conversely, every prime
$N \equiv \pm 1 \pmod 8$ is so represented:  the discriminant $8$ has one class, and the
fundamental unit $1+\sqrt2$ of $\mathbb{Z}[\sqrt2]$ has norm $-1$, so the forms
$m^2-2k^2$ and $2k^2-m^2$ represent the same primes, namely those for which $2$ is a
quadratic residue.  The area equation therefore excludes no prime in that class, and the
following theorem has to be proved by other means.
\end{remark}

\begin{theorem}[$j>1$] \label{theorem:triquadratic}
Let triangle $T$ have angles  $(2\alpha,\beta,\alpha+\beta)$, and suppose
$T$ is $N$-tiled by a tile with $(\alpha,\beta,\gamma)$, and
$\alpha$ not a rational multiple of $\pi$.
  Then $N$ is
not prime, indeed $N$ is not even squarefree.  In fact,
$$ N = j^2(2k^2-m^2),$$
where $j > 1$ and $(a,b,c) = (mk,k^2-m^2,k^2)$.
\end{theorem}

\begin{proof}
Suppose, for proof by contradiction, that there is such a tiling (with $j=1$) of
a triangle $ABC$ of the shape $(2\alpha,\beta,\alpha+\beta)$,  with
the $2\alpha$ angle at $A$ and the $\beta$ angle at $B$.  Then,
tacking a quadratic tiling onto side $AC$ as described in Lemma~\ref{lemma:trisector}
and illustrated%
\footnote{ $AC$ is the side of length $kb$ between the $2\alpha$ vertex and the $\alpha+\beta$
vertex. In Figure~\ref{figure:44} it is the left trisector; many drawings of triquadratics
are oriented with that side horizontal, but not in Figure~\ref{figure:44} where it occurs as
part of an isosceles tiling.
} in Figure~\ref{figure:44},
we construct a $j=1$ tiling of the isosceles triangle of shape $(\beta,\beta,3\alpha)$.
But that contradicts Theorem~\ref{theorem:isoscelesbeta}.
\end{proof}

\section{The prime \Erdos~634 problem}

\begin{theorem} \label{theorem:prime634}
Let $N$ be prime and greater than 3.  Suppose
triangle $ABC$ can be $N$-tiled.  Then 
\begin{itemize}
\item  $N \equiv 1 \pmod 4$.
\item  $ABC$ is similar to the tile.
\item  Both the tile and $ABC$ are similar to a right triangle with legs $e$ and $f$, where $N = e^2 + f^2$.
\end{itemize}
\end{theorem}

\begin{proof}
 Suppose $N$ is prime, $N > 3$, $N \equiv 3 \pmod 4$, and some triangle $T$ is $N$-tiled.
If $T$ is equilateral, Theorem~\ref{theorem:equilateral} applies.  If $T$ is isosceles,
then by Corollary~\ref{cor:isoprime} it has the shape $(\beta,\beta,3\alpha)$, excluded by
Theorem~\ref{theorem:isoscelesbeta}.  If $T$ is not isosceles and has commensurable
angles, the tiling is a reptiling by Theorem~\ref{theorem:l1995-5.3}.  Otherwise,
by Theorem~\ref{theorem:l1995-4.1} the tiling is a reptiling, or a Group~1 or a Group~2
tiling.  The Group~2 shapes are settled:  the isosceles one in \cite{bonfioli2026}, the
other four in Theorems~\ref{theorem:case2}, \ref{theorem:case3}, \ref{theorem:case4}
and~\ref{theorem:case5}.  Of the five Group~1 shapes, three are settled in
\cite{beeson-triangletiling3} as listed in Table~\ref{table:Group1};  the remaining two
are the triquadratic and isosceles-$\beta$ shapes, where $N$ prime forces $j=1$, excluded
for the isosceles-$\beta$ shape by Theorem~\ref{theorem:isoscelesbeta},  and hence for
the triquadratic shape by Theorem~\ref{theorem:triquadratic}.  So the tiling is a
reptiling, and by Theorem~\ref{theorem:reptile}, $N$ is a square, a sum of two squares, or
three times a square.  As $N > 3$ is prime, it is neither a square nor three times a
square, so it is a sum of two squares and $N \equiv 1 \pmod 4$, and Theorem~\ref{theorem:reptile}
gives the shape of the tile as stated in the theorem.
\end{proof}

\begin{corollary} \label{corollary:prime634cor}
Let $N$ be prime.  Then some triangle can be $N$-tiled if and only if 
$N = 2$, $N = 3$, or $N \equiv 1 \pmod 4$.
\end{corollary}

\begin{proof}  
Any isosceles triangle can be 2-tiled, and any equilateral triangle can be 3-tiled,
so we may suppose $N > 3$.  

{\em Left to right}.  Suppose $ABC$ is $N$-tiled.  By Theorem~\ref{theorem:prime634},
$N \equiv 1 \pmod 4$.

{\em Right to left}.  Suppose $N \equiv 1 \pmod 4$.  Then $N = e^2+f^2$
for some positive integers $e,f$, and the right triangle with legs $e$ and $f$ is
reptiled by $N$ copies of itself:  the altitude to the hypotenuse splits it into two
similar copies, which are subdivided quadratically into $e^2$ and $f^2$ tiles.  See
Figure~\ref{figure:tiling13} for the case $N=13$.
\end{proof}

\section{Summary}
The question that concerns us is which primes $N$ admit an $N$-tiling of some triangle.
 We have answered that question in full:  only $N=2$, $N=3$, and $N$ a
prime congruent to $1 \pmod 4$  are possible, and all of them occur.
  In particular,
there are no tilings with prime $N$ congruent to $3 \pmod 4$ and greater than $3$.
Moreover, for tilings with a tile $(\alpha, \beta, 2\pi/3)$,  we found several tilings
with squarefree $N$, showing that the theorem does not generalize from prime to squarefree.
However, in the newly-settled Group~1 cases,  we did also prove that $N$ cannot be
squarefree; in fact more generally, we cannot have $j=1$,  where the formula for
$N$ has the form $j^2 f(a,b)$.  These results completely settle the prime case
of \Erdos\ problem 634,  which is to characterize the (prime) $N$ such that some
triangle has an $N$-tiling.

\section*{Use of AI}
The first author used Claude (Anthropic) for proof-checking and copy-editing.
The second author used Claude (Anthropic) and ChatGPT (OpenAI, through its Codex agents)
to produce the three new squarefree tilings, and our proofs of Lemma~\ref{lemma:claude3},
Lemma~\ref{lemma:claude4}, Lemma~\ref{lemma:not3divc}, and
Theorem~\ref{theorem:isoscelesbeta} and the lemmas leading up to it.
  The first author has presented these proofs here in his
own words. We prompted the AIs and we remain accountable for the work.

The formulas for $N$ in the Group~2 cases were already in \cite{zhang2026},
but were rediscovered by Claude without looking at  \cite{zhang2026}.
The first author used Claude to formalize the entire paper in Lean.

\appendix
\section{Formal Verification in Lean}\label{section:lean}

The results of this paper have been formalized in the Lean~4 proof assistant
(version~4.30.0), using the Mathlib library.  The Lean code is at
\url{https://github.com/mjbeeson/TriangleTilingPrimeLean}, release v1.0.
Each theorem, lemma and corollary listed in Table~\ref{table:lean}, and the definition of a
run, has a Lean counterpart, named after its label, whose proof is checked by the Lean
kernel using only the standard axioms ({\tt propext}, {\tt Classical.choice},
{\tt Quot.sound}), with no {\tt sorry} and no {\tt native\_decide}.
All Lean names are in the namespace {\tt TriangleTiling.Prime}, so labels that also occur
in \cite{beeson-triangletiling3} (for example {\tt angles} and {\tt twocs}) name different
Lean theorems; labels containing {\tt -} or {\tt .} are written with French quotes in Lean,
as in $\ll${\tt l1995-5.3}$\gg$.
Tilings are formalized as in V.~Bonfioli's Lean library \cite{bonfioli-lean}:
$N$ triangles congruent to a given tile (reflections allowed), with disjoint
interiors, whose union is $T$.

\smallskip
\noindent{\em Results cited from other papers.}
The results of \cite{beeson-triangletiling3}, \cite{beeson-isosceles} and
\cite{beeson-equilateral} that are cited here are taken from our Lean formalizations of
those papers.  Laczkovich's classification of the tilings of a triangle by congruent
tiles (Theorems~4.1, 5.1 and 5.3 of \cite{laczkovich1995} and Theorem~3.3 of
\cite{laczkovich2012}), Theorem~1.1 of \cite{beeson-zhang2026}, the theorem of Snover
et al.\ (Theorem~\ref{theorem:reptile}, necessity only), Theorem~42 of
\cite{bonfioli2026}, and the tilings constructed in the proof of
Corollary~\ref{corollary:prime634cor} are taken from \cite{bonfioli-lean}.
Theorem~\ref{theorem:rationality} is assembled from these, case by case.
The results marked $\dagger$ in Table~\ref{table:lean} depend on the classification and
are in a separate part of the Lean code, which takes several hours to build.
\begin{table}[htbp]
\small
\caption{Results  and their Lean names (all in the namespace
{\tt TriangleTiling.Prime}; $\dagger$: in the separately built part).}
\label{table:lean}
\begin{tabular}{ll@{\qquad}ll}
Result & Lean name & Result & Lean name \\
\hline
Lemma~\ref{lemma:RandT} & {\tt RandT}$^\dagger$ & Lemma~\ref{lemma:angles} & {\tt angles} \\
Theorem~\ref{theorem:rationality} & {\tt rationality}$^\dagger$ & Lemma~\ref{lemma:nogammacorner} & {\tt nogammacorner} \\
Theorem~\ref{theorem:equilateral} & {\tt equilateral}$^\dagger$ & Lemma~\ref{lemma:anglesOK} & {\tt anglesOK} \\
Theorem~\ref{theorem:l1995-5.3} & {\tt l1995-5.3}$^\dagger$ & Lemma~\ref{lemma:gsquare} & {\tt gsquare} \\
Theorem~\ref{theorem:reptile} & {\tt reptile}$^\dagger$ & Lemma~\ref{lemma:alphaoverpi} & {\tt alphaoverpi} \\
Corollary~\ref{corollary:reptilenotprime} & {\tt reptilenotprime}$^\dagger$ & Lemma~\ref{lemma:lattice} & {\tt lattice} \\
Theorem~\ref{theorem:isosceles} & {\tt isosceles}$^\dagger$ & Theorem~\ref{theorem:minlengths} & {\tt minlengths} \\
Corollary~\ref{cor:isoprime} & {\tt isoprime}$^\dagger$ & Lemma~\ref{lemma:edgedecompositions} & {\tt edgedecompositions} \\
Theorem~\ref{theorem:l1995-4.1} & {\tt l1995-4.1}$^\dagger$ & Lemma~\ref{lemma:twocs} & {\tt twocs} \\
Lemma~\ref{lemma:c120} & {\tt c120} & Definition~\ref{definition:run} & {\tt run} \\
Lemma~\ref{lemma:claude} & {\tt claude} & Lemma~\ref{lemma:block} & {\tt block} \\
Lemma~\ref{lemma:beta} & {\tt beta} & Lemma~\ref{lemma:rootedblock} & {\tt rootedblock} \\
Lemma~\ref{lemma:pairwisecoprimality} & {\tt pairwisecoprimality} & Lemma~\ref{lemma:longrun} & {\tt longrun} \\
Lemma~\ref{lemma:not3divc} & {\tt not3divc} & Lemma~\ref{lemma:isobetasides} & {\tt isobetasides} \\
Lemma~\ref{lemma:BLZ30} & {\tt BLZ30} & Lemma~\ref{lemma:isoscelesnob} & {\tt isoscelesnob} \\
Lemma~\ref{lemma:claude4} & {\tt claude4} & Theorem~\ref{theorem:isoscelesbeta} & {\tt isoscelesbeta}$^\dagger$ \\
Lemma~\ref{lemma:claude3} & {\tt claude3} & Lemma~\ref{lemma:gammagreaterpiover2} & {\tt gammagreaterpiover2} \\
Lemma~\ref{lemma:rationality3} & {\tt rationality3}$^\dagger$ & Lemma~\ref{lemma:sidelengths} & {\tt sidelengths} \\
Theorem~\ref{theorem:case2} & {\tt case2}$^\dagger$ & Lemma~\ref{lemma:trisector} & {\tt trisector} \\
Theorem~\ref{theorem:case3} & {\tt case3}$^\dagger$ & Theorem~\ref{theorem:triquadratic} & {\tt triquadratic}$^\dagger$ \\
Theorem~\ref{theorem:case4} & {\tt case4}$^\dagger$ & Theorem~\ref{theorem:prime634} & {\tt prime634}$^\dagger$ \\
Theorem~\ref{theorem:case5} & {\tt case5}$^\dagger$ & Corollary~\ref{corollary:prime634cor} & {\tt prime634cor}$^\dagger$ \\
\end{tabular}
\end{table}
\FloatBarrier

\begingroup
\raggedright
\bibliographystyle{plainurl}

\end{document}